\pdfoutput=1
\documentclass[12pt,a4paper]{article}
\usepackage{authblk}
\usepackage[sort&compress,numbers]{natbib}
\usepackage{fullpage}
\usepackage{latexsym}
\usepackage{epsfig}
\usepackage{subfigure}
\usepackage{amsfonts}
\usepackage{amssymb}
\usepackage{amsmath}
\usepackage{amsthm}
\usepackage{url}
\usepackage{algorithm}
\usepackage{algorithmic}
\usepackage{graphicx}
\usepackage[utf8]{inputenc}
\usepackage[font=small]{caption}

\newtheorem{theorem}{Theorem}[section]
\newtheorem{lemma}[theorem]{Lemma}

\theoremstyle{remark}

\theoremstyle{definition}

\newcommand\Nvec[1]{{\bf #1}}
\newcommand\dvec[1]{{#1}}
\DeclareMathOperator*{\myspan}{span}

\newcommand\myD{D}   
\newcommand\myDelta{\Delta}  
\newcommand\myomega{\omega}  
\newcommand\myw{w}  
\begin{document}

\title{An Adaptive Multiscale Approach for Electronic Structure Methods}
\author[1]{Sambasiva Rao Chinnamsetty}
\author[1,2]{Michael Griebel}
\author[2]{Jan Hamaekers}
\affil[1]{{\small Institut f\"ur Numerische Simulation, Universit\"at Bonn, Wegelerstr. 6, D-53115 Bonn}}
\affil[2]{{\small Fraunhofer-Institut f\"ur Algorithmen und Wissenschaftliches Rechnen SCAI, Schloss Birlinghoven, D-53754 Sankt Augustin}}

\maketitle

\begin{abstract}
\noindent
In this paper, we introduce a new scheme for the efficient numerical
treatment of the electronic Schr\"odinger equation for molecules. It
is based on the combination of a many-body expansion, which
corresponds to the bond order dissection Anova approach introduced in
\cite{griebel2014bond,Heber2014}, with a hierarchy of basis sets of
increasing order. Here, the energy is represented as a finite sum of
contributions associated to subsets of nuclei and basis sets in a
telescoping sum like fashion. Under the assumption of data locality of
the electronic density (nearsightedness of electronic matter), the
terms of this expansion decay rapidly and higher terms may be
neglected.  We further extend the approach in a dimension-adaptive
fashion to generate quasi-optimal approximations, i.e.\
a specific truncation of the hierarchical series such that the total
benefit is maximized for a fixed amount of costs. This way, we are
able to achieve substantial speed up factors compared to conventional
first principles methods depending on the molecular system under
consideration. In particular, the method can deal efficiently with
molecular systems which include only a small active part that needs to
be described by accurate but expensive models.
\end{abstract}

\section{Introduction}
The idea of so-called QM/MM hybrid approaches is to combine highly
accurate quantum mechanical (QM) methods and fast molecular mechanics
(MM) methods in a cost efficient manner \cite{andre2014nobel}. Such methods make use of the fact that, in
many applications, it is sufficient to model a small part of a system
in great detail and the rest of the system in less detail only
\cite{senn2009qm,praprotnik2008multiscale}.
Numerically, the microscale with its reactive part is usually treated
with QM approaches like e.g.\ Hartree-Fock (HF), configuration interaction (CI), M\"oller-Plesset
(MP2), coupled cluster (CC) or density functional theory (DFT)
methods which yield approximate solutions to the underlying
quantum-mechanical electronic Schr\"odinger equation. The
mesoscale with its non-reactive part is described by classical MM
methods which use Newton’s mechanics with empirically fitted potential
functions. Here, one of the main challenges is to define the QM region,
the MM region and the interactions between
them~\cite{andre2014nobel}.  The ultimate goal would be a seamless
coupling of QM computations where needed and classical MM simulations
where sufficient. Such approaches are usually referred to as
multiscale methods~\cite{nakano2001multiscale,cances2008molecular,cances2003computational,le2005computational,lubich2008quantum}.\footnote{Note that the 2013 Nobel Prize in chemistry was
  awarded to Karplus, Levitt and Warshel for {\em the development of
    multiscale models for complex chemical systems}
  \cite{andre2014nobel}.}

In this article, we introduce a new method for a seamless coupling of
different models.  In our so-called {\em adaptive multilevel
  BOSSANOVA} approach, we follow the idea of general sparse grids to
combine, on the one hand, an appropriate many-body expansion and, on
the other hand, a hierarchy of models in a cost-efficient manner. The
so-called regular sparse grid combination
technique~\cite{Bungartz.Griebel.Roschke.ea:1996,Griebel.Harbrecht:2014}
is well-known from high-dimensional integration, interpolation and the
solution of elliptic
PDEs~\cite{Gerstner.Griebel:1998,bungartz2004sparse}. More generally,
the construction of a sparse grid can be formulated as a knapsack
problem~\cite{bungartz2004sparse}. Here, the benefit and the cost of
each hierarchical surplus in an appropriate hierarchical
representation is estimated. Then, a quasi-optimal sparse grid
approximation can be
achieved~\cite{bungartz2004sparse,nobile2015convergence} by a proper
truncation of the hierarchical expansion such that the total benefit
is maximized for a given total work load. This procedure can be applied
using a priori estimates or a posteriori estimates, like in dimension-adaptive sparse grid approaches~\cite{gerstner2003dimension}.

After the presentation of our new approach in a general setting, we
discuss the case of the approximation of the Born-Oppenheimer energy
in more detail. To this end, we use the systematic bond order
dissection Anova (BOSSANOVA) algorithm proposed in
\cite{griebel2014bond,Heber2014} and consider a hierarchy of models
which results from the use of an appropriate systematically convergent
hierarchy of one-electron basis sets, like e.g.\ cc-pVDZ, cc-pVTZ,
cc-pVQZ, cc-pV5Z, cc-pV6Z, \ldots, within a specific QM method like
e.g. HF or DFT. We discuss both, a priori truncation schemes and
dimension-adaptive algorithms which are based on a local cost model
and a posteriori local benefit estimators. We apply our new method to
approximate the energy of several molecules, where we are able to
achieve substantial speed up factors compared to the conventional
electronic structure method.  In addition, we apply our new adaptive
approaches to three large molecules, which can not be treated with any
conventional method at all in a reasonable time.

The remaining article is organized as follows.  In
Section~\ref{sec:AddModel} we briefly summarize the basics of the
underlying Schr\"odinger equation and shortly review
  additive models. In Section~\ref{sec:AdaptiveMultilevelBOSSANOVA}
we describe our new multilevel Anova-like decomposition scheme. In
Section~\ref{sec:NumRes} we give numerical results for a broad range
of organic molecules. We conclude with some remarks in Section~\ref{sec:Con}.

\section{Additive model approaches}
\label{sec:AddModel}
In general, any starting point for an approximation or a coupled model
must be the full Schr\"odinger equation for the electrons and nuclei
of the system under consideration. But since the time-dependent
Schr\"odinger equation lives in $3(M+N)+1$ dimensions, where $M$
denotes the number of nuclei and $N$ denotes the number of electrons,
a direct numerical treatment is impossible due to the curse of
dimension. Thus one has to resort to model approximations.  As a first
step, in the Born-Oppenheimer molecular dynamics approach, the wave
functions of the nuclei and electrons are separated, the subsystem of
the nuclei is treated classically with Newton's mechanics and the
remaining $3N$-dimensional electronic Schr\"odinger equation is
further approximated by one of the aforementioned QM methods~\cite{lubich2008quantum}. The
potential needed for Newton's mechanics is obtained from the
electronic solution by the Hellmann-Feynman theorem. Then, the resulting
equations of motion for the degrees of freedom of the nuclei, i.e.\
their positions $\Nvec{R}$, read as follows:
\begin{align}
  m_k\ddot{\Nvec{R}}_k(t) = -\nabla_{\dvec{R}_k}  V^{BO}(\Nvec{Z},\Nvec{R}(t)),
\label{math:NE}
\end{align}
\begin{align}
  \label{math:full-SE}
  V^{BO}({\Nvec{Z}},\Nvec{R}(t)) := \inf_{\|\phi\|=1} \left\{ \int
  \phi^\ast(\Nvec{r}){H}_{e}(\Nvec{Z},\Nvec{R}(t),
  \Nvec{r}) \phi(\Nvec{r})d\Nvec{r}
\right\}.
\end{align}
Here $V^{BO}$ denotes the Born-Oppenheimer ground state energy
and $H_e$ denotes the electronic Hamiltonian which reads
as
\begin{align*}
  {H_e}(\Nvec{Z},\Nvec{R}, \Nvec{r}) &:=  {H_{ee}}(\Nvec{Z}; \Nvec{R}, \Nvec{r})
  +\sum_{i<j}^M\frac{Z_{i}Z_{j}}{\mid\mid
      \dvec{R}_i-\dvec{R}_j\mid\mid},\\
  {H_{ee}}(\Nvec{Z}, \Nvec{R}, \Nvec{r}) &:=
-\frac{1}{2}\sum_{k=1}^N\Delta_{\dvec{r}_k}+ \sum_{k<l}^N\frac{{1}}{\mid\mid
      \dvec{r}_k-\dvec{r}_l\mid\mid}
    \!-\!\sum_{k=1}^N\sum_{j=1}^M\frac{Z_{j}}{\mid\mid
      \dvec{r}_{k}-\dvec{R}_j\mid\mid},
\end{align*}
where $\Nvec{Z}=(Z_1, \ldots, Z_M)$) denotes the atomic
numbers and $\Nvec{r}$ are the positions of the electrons.
Here, all information on the dynamics of
an atomic system is encoded in the associated high-dimensional
Born-Oppenheimer potential $V^{BO}$.  Note that for neutral or
positively charged systems, i.e.\ $N \leq \sum_{j=1}^{M} Z_j$,
Zhislin’s theorem states that $V^{BO}(\Nvec{Z}, \Nvec{R})$ is an
isolated eigenvalue of finite multiplicity of the operator
$H_{ee}(\Nvec{Z}, \Nvec{R}, \Nvec{r})$ for all $\Nvec{R}$
\cite{friesecke2003multiconfiguration}. Note furthermore that
difficulties may in general arise for the Born-Oppenheimer
ground-state molecular dynamics from the fact that
$\nabla_\Nvec{R} V^{BO}$ is unbounded (because of Coulomb
singularities at nuclei cusps) and from the fact
that $\nabla_\Nvec{R} V^{BO}$ might be discontinuous even away from
the positions of Coulomb singularities (because of possible eigenvalue
crossings of the electronic Hamiltonian $H_{ee}$)
\cite{ambrosio2011semiclassical}. In the following, we will omit the
parameter $\Nvec{Z}$ if it is clear from the context.

Let us remark that a global electronic QM solution is, at least for
larger molecules, still too expensive since conventional methods scale at
best with ${O}(M^3)$ due to the underlying problem of matrix
diagonalization. Therefore, specific electronic structure methods
\cite{ Goedecker1999, Skylaris2005, bowler2012mathcal} are employed
which scale linearly and thus overcome this complexity problem. There exist for example divide and conquer
DFT~\cite{yang1995density,shimojo2008divide,wang2014divide} and
partition DFT (PDFT)~\cite{elliott2010partition}. Such approaches are
based on a partition of the global system into embedded spatial local
parts~\cite{huang2011quantum}, which can be solved separately and then
combined to approximate the global solution.  This is also the
underlying idea of various decomposition and fragmentation approaches,
where the full global electronic structure problem is decomposed into
appropriate local subproblems, while the local results are linearly combined to
generate a consistent energy expression for the global
system~\cite{gordon2012fragmentation}. For example, there exist the
sum of interactions between fragments computed ab initio
procedure (SIBFA)~\cite{GreshClaverie1984}, the fragmentation reconstruction
method (FRM)~\cite{Amovilli2002}, the fragment molecular orbital
method (FMO)~\cite{Kitaura1999,komeiji2003fragment,mayhall2011molecules}, additive model
approaches~\cite{Deev2005,Collins2006} and many-body
expansions~\cite{DawBaskes1984,Abell1985,Tersoff1989,bygrave2012embedded,richard2012generalized,richard2014understanding,lao2016understanding,gora2011interaction,dahlke2007electrostatically}. Note
that the linear scaling QM methods and fragmentation based QM methods
all take advantage of a data locality principle which involves the
so-called nearsightedness of electronic
matter~\cite{kohn1996density,prodan2005nearsightedness}.

Although these QM based methods are very powerful, they are still too
costly for many applications, where large systems have to be treated
with high accuracy. Hence, in order to further reduce the costs without
losing much accuracy, one tries to somehow use different levels of
models (e.g.\ HF and CCSD(T)) in a decomposition approach.

For example in {\em incremental methods}, a many-body expansion
\begin{align*}
  E(R_1, \ldots, R_M) = F_\emptyset +
  \sum_{i}F_{\{i\}}(R_i)+\sum_{i<j}F_{\{i,j\}}(R_i, R_j) + \ldots +
  F_{\{1,\ldots,M\}}(R_1, \ldots, R_M),
\end{align*}
can be used to couple different models. Here, the leading
low-order terms in the many body expansion are treated with a {\em
  high} level method, while all higher-order terms are treated at a
{\em low} level model~\cite{wen2012practical}, like e.g.\ the
multilayer hybrid approach~\cite{beran2009approximating}, the
multilevel fragment-based scheme~\cite{rezac2009multilevel} or the
many-body integrated fragmentation technique~\cite{tschumper2006multicentered,bates2011development}.

There exist several methods, which apply different models to different
parts of an atomic system, while combining the results to produce a
consistent energy expression. This is the case for example in the {\em
  conventional QM/MM} methods \cite{field1990combined} and several
variants like e.g.\ the IMOMM ansatz~\cite{MaserasMorokuma1995} and
the ONIOM approach~\cite{Vreven2000}. The common basic idea of most
hybrid QM/MM approaches is to use a sum over two nested regions, like
e.g.\ $\Omega_2$ and $\Omega_1 \subseteq \Omega_2$, to represent the
total energy as
\begin{align}\label{equ:QMMM}
  {E}_{QM}({\Omega_2}) \approx {E}_{MM}({\Omega_2}) +
  \underbrace{({E}_{QM}({\Omega_1}) -
    {E}_{MM}({\Omega_1}))}_{\text{model improvement}}.
\end{align}

While all these hybrid methods and fragmentation
procedures have promising features, they involve stringent chemical
knowledge to choose the regions (or cuts) as best as possible while
keeping the underlying ground-state electronic density intact.

Furthermore, in {\em quantum chemistry composite methods}
\cite{ohlinger2009efficient,deyonker2006correlation} an additive model
approximation is used to couple different levels of models and
different levels of basis sets in a cost-efficient manner. Here, for a
low level method $E_1$, a high level method $E_2$, a small basis set
$B_1$ and a large basis set $B_2$, the energy of the high level method
using the large basis, i.e.\ $E_2(B_2)$, is approximated
by the sum
\begin{align}\label{equ:QCCM}
  E_2(B_2) &\approx
  E_1(B_1)+\underbrace{(E_1(B_2)-E_1(B_1))}_{\text{basis set
      improvement}} + \underbrace{(E_2(B_1)-E_1(B_1))}_{\text{model improvement}}\\\nonumber
  &= E_1(B_2)+ (E_2(B_1)-E_1(B_1)).
\end{align}
This approximation provides in many cases a substantial reduction of
computational costs but results in just a slight degradation of
accuracy~\cite{martin1999ab}.

Let us remark that the additive approximations (\ref{equ:QMMM}) and
(\ref{equ:QCCM}) can be used as basic building blocks which lead in a
recursive fashion formally to the so-called sparse grid combination
technique~\cite{Bungartz.Griebel.Roschke.ea:1996,Griebel.Harbrecht:2014},
which is the basis for our new adaptive multilevel BOSSANOVA approach.

\section{Adaptive multilevel bond order dissection Anova}
\label{sec:AdaptiveMultilevelBOSSANOVA}
First of all, let us note that, in many typical applications, it is
sufficient to accurately describe just a restricted specific subdomain
of the Born-Oppenheimer surface. Therefore, we assume that the error
of the quantity of interest of an approximation $\tilde{V}$ to the
Born-Oppenheimer surface $V^{BO}$ can be described by
$|\mathcal{E}[V^{BO}-\tilde{V}]|$, where $\mathcal{E}$ is an
appropriate linear functional. Examples would be energies,
forces or Hessians at specific coordinates or on specific subdomains.  In
this paper we restrict our numerical experiments to the case of the
energy at a specific fixed configuration $\Nvec{R}$. To this end, one
can define $\mathcal{E}[V]$ just by the convolution of $V$ and the delta
distribution $\delta$, i.e.
\begin{align*}
  \mathcal{E}[V](\Nvec{R}) := \int
  V(\Nvec{R}^\prime)\delta(\Nvec{R}-\Nvec{R}^\prime)\, d\Nvec{R}^\prime = V(\Nvec{R})
\end{align*}

Our new approach is based on the conventional BOSSANOVA
method~\cite{griebel2014bond,Heber2014} and on a hierarchy of models
resulting from sequences of correlation-consistent basis sets that are
well-known in
quantum-chemistry~\cite{dunning1989gaussian,feller1993use}. These
methods are coupled using the idea of the optimized sparse grid
technique by balancing benefit-cost ratios of hierarchical
surpluses, for details, see
\cite{bungartz2004sparse,nobile2015convergence}.

In the following subsections, we will first discuss the so-called
complete basis set limit, then recall the conventional BOSSANOVA
scheme and finally describe our new multilevel and adaptive multilevel
BOSSANOVA schemes. Moreover, we will discuss the estimation of local
benefit and local cost associated to a hierarchical surplus
whose quotient serves as error indicator in the adaptive scheme.

\subsection{Complete basis set limit}
\label{sec:CBS}
Now let us shortly recall the {\em full configuration interaction} (FCI)
method, which corresponds to the application of the Galerkin scheme
for the numerical treatment of the electronic Schr\"odinger equation
(\ref{math:full-SE}). To this end, let $B_1 \subseteq B_2 \subseteq
\ldots$ be a sequence of appropriate finite basis sets of one-electron
functions in $\mathcal{H}^1(\mathbb{R}^3\times\{\pm\tfrac{1}{2}\})$
such that the family $\left\{\myspan\{B_p\}\right\}_{p\in\mathbb{N}}$
is dense in $\mathcal{L}^2(\mathbb{R}^3\times\{\pm\tfrac{1}{2}\})$.\footnote{Except
  for the completion with respect to a chosen Sobolev norm,
  $\sum_{p\in\mathbb{N}}\myspan\{B_p\}$ is just the associated Sobolev
  space $\mathcal{H}^1(\mathbb{R}^3\times\{\pm\tfrac{1}{2}\})$.}
Then, starting from such a family of one-electron basis functions,
Slater determinants (which are generated by the outer anti-symmetric
product $\wedge$) are used to construct a family of $N$-electron bases
functions and an associated dense
family of finite dimensional $N$-electron subspaces
$\{\mathcal{V}_p\}_{\in\mathbb{N}}$ for the $N$-electron space
\begin{align*}
  \mathcal{V} :=
  \mathcal{H}^1(\mathbb{R}^{3}\times\{\pm\tfrac{1}{2}\})^N \cap
  \bigwedge_{i=1}^N
  \mathcal{L}^2(\mathbb{R}^{3}\times\{\pm\tfrac{1}{2}\})^N,
\end{align*}
for details,
see~\cite{griebel2007sparse,Hamaekers:2010,schneider2009analysis}. This
leads to an associated family of potential energy functions
$\{V^{FCI}_{p}\}_{p\in\mathbb{N}}$. Here, for a fixed $\Nvec{R}=(R_1,
\ldots, R_M)$ with pair-wise distinct components,
$V^{FCI}_{p}(\Nvec{R})$ is given as
\begin{align}\label{equ:FCI}
  V^{FCI}_{p}(\Nvec{R}) := \inf_{\phi \in \mathcal{V}_p, \|\phi\|=1} \int \phi(\Nvec{r})H_e(\Nvec{R}; \Nvec{r}) \phi(\Nvec{r})
  d\Nvec{r},
\end{align}
which results from the discretization of $\mathcal{V}$ by
$\mathcal{V}_p$, i.e.\ from a Galerkin-Ritz approximation of the exact
ground-state energy $V^{BO}(\Nvec{R})$ of (\ref{math:full-SE}).

In particular, the FCI method fulfills the Rayleigh-Ritz variational
principle and thus it holds $V^{BO}(\Nvec{R}) \leq
V^{FCI}_{p}(\Nvec{R})$. Furthermore, if the exact lowest eigenvalue
$V^{BO}(\Nvec{R})$ with associated eigenfunction $\phi^{SE}$ has
multiplicity one, then there exist a $p_0\in\mathbb{N}$ and functions
$\phi^{FCI}_p(\Nvec{R})\in\mathcal{V}_p$ which solve the discrete
Galerkin eigenvalue problem (\ref{equ:FCI}) and obey the following
error bounds
\begin{align*}
  \|\phi^{BO}(\Nvec{R})-\phi^{FCI}_p(\Nvec{R})\|_{\mathcal{H}^1} \leq C_1(\Nvec{R}) \inf_{\phi \in \mathcal{V}_p, \|\phi\|=1}\|\phi^{BO}(\Nvec{R})-\phi\|_{\mathcal{H}^1},\\
  0 \leq V^{FCI}_p(\Nvec{R}) - V^{BO}(\Nvec{R}) \leq C_2(\Nvec{R})
  \|\phi^{BO}(\Nvec{R})-\phi^{FCI}_p(\Nvec{R})\|_{\mathcal{H}^1}^2
\end{align*}
for all $p >p_0$, see e.g.\ \cite{griebel2007sparse,
  schneider2009analysis,yserentant2010regularity}, where
$C_1(\Nvec{R})$ and $C_2(\Nvec{R})$ denote constants which are
independent of $p$. Moreover, since $\phi^{BO}$ is at least in
$\mathcal{H}^1$ and fulfills certain decay properties, the use of an
appropriate sequence of bases of sufficient order leads to an error
estimate of type
\begin{align*}
  V_p^{FCI}(\Nvec{R})-V^{BO}(\Nvec{R}) \leq C(\Nvec{R})p^{-2},
\end{align*}
compare e.g.\ \cite{yserentant2010regularity,zeiser2012wavelet}.
In the following, we will assume that,
for all $\epsilon >0$, there exists a $p_\epsilon\in\mathbb{N}$ with
\begin{align*}
  \left|\mathcal{E}[V^{BO}- V^{FCI}_{p_\epsilon}]\right| < \epsilon
\end{align*}
and thus also the convergence of the quantity of interest
$\mathcal{E}[V^{BO}]$ is provided.  Unfortunately, the applicability
of the FCI method is limited by the curse of dimensionality, since the
number of involved degrees of freedom grows exponentially with the
number of electrons $N$. Using the specific regularity of the
electronic eigenfunction and provided that its decay properties are
present, this curse of dimension can be circumvented to a certain
extent by sophisticated sparse grid techniques. But these methods are
still restricted to small systems
\cite{yserentant2010regularity,griebel2010tensor,kreusler2012mixed,yserentant2014regularity}.

Thus, we more generally consider electronic structure methods (ESM)
which are also based on a discretization of the one-electron space
$\mathcal{H}^1(\mathbb{R}^3\times\{\pm\tfrac{1}{2}\})$, but represent
model approximations to the electronic Schr\"odinger equation, like in
particular HF, CI, CC and DFT. Here, an appropriately chosen
family $\{B_p\}_{p\in\mathbb{N}}$ of one-electron basis function sets
again results in a family of associated potential energy functions
$\{V^{ESM}_p\}_{p\in\mathbb{N}}$.  We furthermore assume that the telescopic series
\begin{align*}
  V^{ESM}_{\infty} := \sum_{p=1}^{\infty} \left(V^{ESM}_{p}-V^{ESM}_{p-1}\right),
\end{align*}
is point-wise absolute convergent. Then, we can estimate the error of
an approximation $V^{ESM}_{p}$ by
\begin{align}\label{equ:ltelescopicsum}
  \left|\mathcal{E}[V^{BO}-V^{ESM}_{p}]\right| \leq  \left|\mathcal{E}[V^{BO}-V^{ESM}_{\infty}]\right|+\left|\mathcal{E}[V^{ESM}_{\infty}-V^{ESM}_{p}]\right|,
\end{align}
where
$\left|\mathcal{E}[V^{BO}-V^{ESM}_{\infty}]\right|$
is the model error and
$\left|\mathcal{E}[V^{ESM}_{\infty}-V^{ESM}_{p}]\right|$
is the discretization error, respectively. In the following, we will
just consider the approximation of the so-called complete basis set
limit $V^{ESM}_{\infty}$ for a given electronic structure method.

To this end, let us shortly review complete basis set limit
extrapolation schemes which are well known in computational chemistry
\cite{jensen2013introduction}. For a specific system $\Nvec{R}$ with
pair-wise distinct particle coordinates, the procedure is as follows:
First, a sequence of energies is computed using an appropriate ESM with a
hierarchical sequence of basis sets. Then, these energy values are used
to fit a model equation for the error decay and finally this decay model is applied to
extrapolate to the complete basis set
limit. For example, correlation-consistent
basis sets like e.g.\ cc-pV$n$Z with $n\in\{2, 3, 4, 5, 6,
\ldots\}$ are appropriate.\footnote{Note that $n$ is $2$ for cc-pVDZ, $n$ is $3$ for
  cc-pVTZ, $n$ is $4$ for cc-pVQZ and so on.}
Note that extensive numerical studies
\cite{peterson1994benchmark,halkier1998basis,fast1999infinite,jensen2013introduction}
have demonstrated that in many cases the HF energy converges
asymptotically as $e^{-n}$ and the correlation energy, i.e.\ the
difference between the HF energy and the total energy, converges as $n^{-3}$,
respectively. Here, $n$ is the maximal angular momentum present in the
basis set. Since the total energy is a sum of the HF and the
correlation energy, the extrapolation to the CBS limit $V^{ESM}_\infty(\Nvec{R})=V^{HF}_\infty(\Nvec{R})+V^{corr}_\infty(\Nvec{R})$ can be done
separately for both components \cite{jurevcka2006benchmark}. To this
end, for the HF limit, exponential formulae are popular~\cite{jensen2005estimating}, like e.g.\
\begin{align*}
  V^{HF}_n(\Nvec{R}) = V^{HF}_\infty(\Nvec{R})+A e^{-\alpha n}.
\end{align*}
For the correlation energy, limit formula are common which involve
rational functions~\cite{truhlar1998basis}, like e.g.\
\begin{align*}
  V^{corr}_n(\Nvec{R}) = V^{corr}_\infty(\Nvec{R})+B n^{-\beta},
\end{align*}
often with $\beta=3$. Besides, also other types of formulae are
used for extrapolation which involve, e.g., a mixture of an exponential and a squared
exponential~\cite{peterson2002accurate}
\begin{align*}
  V^{ESM}_n(\Nvec{R}) = V^{ESM}_\infty(\Nvec{R})+a^ne^{-(n-1)}+b^ne^{-(n-1)^2}.
\end{align*}
Analogously to such extrapolation schemes, we assume in the following that
we have a family of potential functions $\{V_p\}_{p\in\mathbb{N}}$
for which it holds
\begin{align*}
  \left|\mathcal{E}[V_{p}-V_{p-1}]\right| \lesssim  g(p)
\end{align*}
for the error of a considered quantity of interest, where the series
$\sum^\infty_pg(p)<\infty$ is absolute convergent.

Note here that $g$ represents an upper limit of the decay behavior of
the so-called hierarchical surplus \cite{bungartz2004sparse} which
comes into play for a telescopic sum expansion of $V_\infty$.
Then, we can estimate the approximation error for $p>0$
from above by
\begin{align*}
  \left|\mathcal{E}[V_p-V_\infty]\right| \lesssim  \sum^\infty_{p^\prime>p}g(p^\prime).
\end{align*}

\subsection{BOSSANOVA}
\label{sec:BOSSANOVA}
Next, let us shortly recall the bond order dissection Anova
(BOSSANOVA) approach as introduced in~\cite{griebel2014bond,Heber2014}.  To this
end, we introduce the notation
\begin{align*}
  V_p(X_1, \ldots, X_M) := V_p(\Nvec{Z},\Nvec{R}), \quad {X}_i := (Z_i, R_i),
\end{align*}
which denotes the energy computed with a specific ESM at level $p$ of
the charge neutral system that consists of $N=\sum^M_i Z_i$ electrons
and $M$ nuclei, each with coordinate vector $R_i \in \mathbb{R}^3$ and
atomic number $Z_i \in \mathbb{N}$. We now decompose the energy $V_p$ into a
multivariate telescopic sum, i.e.\ as a finite series expansion in the
nucleic parameters, in a similar way as in the Anova decomposition.  The
\emph{analysis of variance} (Anova) approach is well-known from
statistics and is closely related to the {\em high-dimensional model
  representation} (HDMR)~\cite{Hayes2006}, the {\em multimode}
approach~\cite{carter1997vibrational} and many-body
expansions~\cite{richard2012generalized}.
This decomposition involves a splitting of the high-dimensional
function into contributions which depend on the positions of single
nuclei and associated charges, of pairs of nuclei and associated
charges, of triples of nuclei and charges, and so on.  Here, we
consider the subset of the nuclei parameters $\{{X}_i\}_{i \in u}$
described by a set of labels $u\subseteq \{1, \ldots, M\}$ with
cardinality $|u|=k$ and call it the \emph{molecular fragment}
associated to $u$ with size $k$.  Then, we consider the many-body
expansion of $V_p$ by
\begin{align}
  V_p({X}_1, \ldots, {X}_M) & =  \myw_{p,\emptyset} + \sum^M_{i_1} \myw_{p,\{i_1\}} ({X}_{\{i_1\}}) + \sum^M_{i_1<i_2} \myw_{p,\{i_1,i_2\}} ({X}_{\{i_1,i_2\}}) + \ldots  \nonumber \\ & \quad \quad \ldots +\myw_{p,\{i_1,\ldots,i_M\}} ({X}_{\{i_1,\ldots,i_M\}}) \nonumber \\
  &=  \sum_{u\subseteq \{1, \ldots, M\}}  \myw_{p,u}({X}_u)\label{equ:Anova_p}
\end{align}
where $X_u$ denotes the set of variables $\{X_i\}_{i \in u}$ and $u
\subseteq \{1,\ldots,M\}$. Each term $\myw_{p,u}$ is
defined by an inclusion-exclusion-type combination of potential
functions that belong to all associated fragments by
\begin{align}\label{equ:Wpu}
  \myw_{p,u}({X}_u) := \sum_{v\subseteq u} (-1)^{|u|-|v|}\widehat{V}_{p,v}({X}_v).
\end{align}
Alternatively it can recursively defined as
\begin{align}\label{equ:Wpu-recursive}
        \myw_{p,u}({X}_u) := \widehat{V}_{p,u}({X}_{u}) - \sum_{k=0}^{|u|-1}\sum_{v\subseteq u, |v|=k} \myw_{p,v}({X}_{v}).
\end{align}
The constant function $\myw_{p,\emptyset}$ is set equal to zero
since it corresponds to the energy of an empty molecular system.
Here, $\widehat{V}_{p,u}({X}_u)$ should be an approximation to the
energy associated with fragment ${X}_u$, where the fixed level $p$
relates to the basis set order used by the ESM, and therefore to the
accuracy it can achieve.
Note at this point that we set
\begin{align}\label{equ:SetVM}
  \widehat{V}_{p,{\{1, \ldots, M\}}}(\Nvec{X}) := V_p(\Nvec{X}),
\end{align}
whereas in general all $\widehat{V}_{p,u}$ with $|u| < M$ could be
chosen arbitrarily since (\ref{equ:Anova_p}) and (\ref{equ:Wpu})
amount to just the identity.  Thus, independent of the specific
definition of $\widehat{V}_{p,u}$ for $|u| < M$, decomposition
(\ref{equ:Anova_p}) is exact and contains $2^M$ different terms due to
the power set construction, i.e.\ $u \subseteq \{1, \ldots, M\}$.

In case of a general splitting it might be that all terms $\myw_{p,u}$
are equally important.  But let us now assume that there is a decay of
the terms $|\mathcal{E}[\myw_{p,u}]|$ with increasing order
$|u|$. Then, a suitable truncation of the sum (\ref{equ:Anova_p})
opens the possibility to avoid the curse of dimensionality.  For
example, if for a small $k<M$ all higher order terms $\myw_{p,u}$ with
$|u|>k$ vanish, only the lower dimensionality of the $k$th order
terms, i.e.\ the effective dimension, exponentially enters the work
count complexity.

It was observed in \cite{griebel2014bond} that, for a proper choice of
the $\widehat{V}_{p,u}$ as good approximation to the ground-state
energy associated to the submolecule $X_u$, the hierarchical surpluses $\myw_{p,u}$ indeed decay for most organic molecules
fast with their
order $k=|u|$. Thus, a proper truncation of the series expansion
(\ref{equ:Anova_p}), e.g.\
\begin{align*}
\sum_{u\subseteq \{1, \ldots, M\}, |u| \leq k }  \myw_{p,u}({X}_u),
\end{align*}
results in a substantial reduction in
computational complexity. We then only have to deal with a sequence of
lower-dimensional subproblems which are associated to the remaining
lower-dimensional energy terms of the decomposition. This leads us to
the following assumption which is central to our further approach: The
associated energy functions $\widehat{V}_{p,u}$ for $|u| < M$ can be
chosen properly such that there is a certain decay in the contribution
of each order $k=|u|$ of the Anova expansion which results in a
monotone convergence of the approximation error with rising
order. Consequently, from a certain order onward, we may neglect the
higher higher-order terms in the Anova decomposition.  Let us remark
that the energy contribution functions $\myw_{p,u}$
in~(\ref{equ:Anova_p}) may be recognized as an expansion of many-body
interaction contributions, as in~\cite{MarxHutter2000}, and thus our
assumption on the decay is also strongly supported by the success of
conventional two- and many-body potential functions used in classical
molecular
dynamics~\cite{Tersoff1989,DawBaskes1984,Brenner2002,Griebel2007a}.

Let us now discuss our specific choice of
$\widehat{V}_{p,u}(X_u)=V_p(X_u)$ according to the BOSS\-ANOVA approach
as presented in \cite{griebel2014bond,Heber2014}, which is aimed at
the approximation of the potential energy surface of a molecule.
A simple choice would be to set $\widehat{V}_{p,u}(X_u)=V_p(X_u)$ for
all $u\subseteq \{1, \ldots, M\}$.  However, the involved cutting of
molecules into fragments may break bonds. Furthermore, a cut-out
fragment may have a total spin unequal zero while the molecular system
itself has a total spin of zero.  This situation would complicate the
proposed linear-scaling ansatz and usually would result in a slow
decay with bond order $k$ of the terms in the many-body expansion
(\ref{equ:Anova_p}). A way to remedy this situation is a saturation of
the dangling bonds of the fragments by adding hydrogen at the places
where bonds were cut, causing the total spin of each augmented
fragment system to be zero. This way, only closed-shell calculations
are performed, which are algorithmically both simpler and more stable.

To be more precise, let $G=(P,E)$ be the graph which is associated to
the organic molecule under consideration and which represents the bond
structure of the molecule.  For reasons of simplicity we assume that
this graph is connected.  Then, a saturation procedure for a
molecular fragment associated to subset $u$ can be described by
$h^G(u)$ additional hydrogen vertices, bonds and their graph-dependent
positions $R^G_i(u)$, $1 \leq i \leq h^G(u)$.  Note here that each
set $u = \{i_1, \ldots , i_{|u|}\}$ of indices of nuclei is
directly associated to an induced subgraph $G_u = ({P}_u, {E}_u)$ of
the total graph $G$ with ${P}_{u} = \{v_{i}\}_{i \in u}$ and $E_{u} =
\{\{v_1, v_2\} \subseteq K \,:\, v_1 \in u, v_2 \in u\}$. Now, for a
subset $u$ with a connected induced subgraph $G_u$, we define a
modified energy function by
\begin{align}\label{equ:saturatedenergy}
  \widehat{V}^{G}_{p,u}({X}_{u}):=V_p\left({X}_{i_1},
  \ldots, {X}_{i_{|u|}}, R^G_1(u), \ldots, R^G_{h^G(u)}(u)\right),
\end{align}
while, according to (\ref{equ:SetVM}), we keep the energy for the
total system unmodified, i.e.\
\begin{align*}
  \widehat{V}^{G}_{p,\{1, \ldots,
  M\}}(\Nvec{X}) := {V}_p(\Nvec{X}).
\end{align*}
This saturation correction is described in detail in
\cite{griebel2014bond,Heber2014}.  In the case of a subset $u$ where
the induced subgraph $G_u$ decomposes into at least two connected
components,, we define the modified energy function
$\widehat{V}^{G}_{p,u}$ as the sum over the modified energy of all
connected components of $G_u$. It follows in particular that the
corresponding hierarchical surplus $\myw_{p,u}$ indeed vanishes. This
is discussed in more detail in Appendix
\ref{sec:AppendixGraph}.\footnote{For example in case of two connected
  (and disjunctive) components $v \cup w = u$ we set
  $\widehat{V}^{G}_{p,u} =
  \widehat{V}^{G}_{p,v}+\widehat{V}^{G}_{p,w}$. Then it follows
  $\myw_{p,u} = \widehat{V}^{G}_{p,v}+\widehat{V}^{G}_{p,w} -
  \widehat{V}^{G}_{p,v} - \widehat{V}^{G}_{p,w} = 0$ by
  Lemma~\ref{lem:Graph}.}  This elimination step is motivated by the
locality of the electronic wave functions: Atoms that share a bond
with a nearby atom will be strongly influenced by changes in the
chemical vicinity of nearest or next-nearest bonding partners whereas
atoms that share no bond to a nearby atom will not.

Altogether, to a given appropriate bond or interaction graph $G$, we
define the conventional BOSSANOVA approximation energy up to order $k$
by
\begin{align}\label{equ:BOSSANOVAENERGY}
  V^{\text{BN}}_{p,k}(\Nvec{X}) := \sum_{k^\prime\leq k} \myD_{p,k}^{G}(\Nvec{X}).
\end{align}
Here, $\myD_{p,k}^{G}$ denotes the hierarchical surplus according to bond order $k$,  which is the sum of
all $\myw_{p,u}^G$ of bond order $k=|u|$ and level $p$, i.e.\
\begin{align}\label{equ:BOSSANOVAENERGY-W}
  \myD_{p,k}^{G}(\Nvec{X}) := \sum_{u\subseteq\{1,\ldots,M\},|u|= k} \myw_{p,u}^{G}({X}_u),
\end{align}
where $\myw_{p,u}^{G}({X}_u)$ is given by (\ref{equ:Wpu}) using the
modified energy functions $\widehat{V}^{G}_{p,u}$ from
(\ref{equ:saturatedenergy}).  Note that in
\cite{griebel2014bond,Heber2014} the conventional BOSSANOVA approach
has been successfully applied to a large range of organic molecules
where in most cases a systematic decay of the size of the decomposition
terms $\myD_{p,k}^G$ with an increase of the order $k=|u|$ has been
observed.  In the following we will omit the parameter $G$ to simplify
notation.

The BOSSANOVA approach is general in the sense that the local
energies can be approximated with any electronic structure method at
hand, e.g.\ HF, CI, CC or MP. This is often referred to as a {\em
  non-intrusive} approach, which means that existing methods and their
implementation can be re-used straightforwardly
without any modifications. Let us finally remark that the
BOSSANOVA method can not be directly applied to metallic systems, since the
necessary assumptions do not hold there, a further
discussion is given in~\cite{Heber2014}.

\subsection{Multilevel BOSSANOVA}
\label{sec:MultBOSSANAOVA}
Now, we describe our new approach to couple the two different
approximation schemes according to (\ref{equ:ltelescopicsum}) and
(\ref{equ:BOSSANOVAENERGY}), respectively. Here, the idea is to
introduce an approximation that is associated to the two corresponding
discretization parameters $p$ and $u$, where for each single parameter, a
systematic improvement of the approximation is expected. Then, we
follow the basic idea of sparse grids~\cite{bungartz2004sparse} and
decompose the associated energy approximation in a telescopic sum like
fashion and finally truncate the expansion in such a way that its
local error and cost contributions are balanced, for details see
\cite{bungartz2004sparse,nobile2015convergence}.

\subsubsection{Hierarchical series expansion}
Up to this point we used the {\em same} basis set of fixed degree $q$
for all BOSSANOVA terms $\myw_{q,u}$. Now we employ electronic structure
methods with basis sets of {\em varying order $p\in\mathbb{N}$} to
approximate the subsystem's ground state energies $\widehat{V}_{p,u}$.
To this end, we further decompose the conventional BOSSANOVA terms
(\ref{equ:Wpu}) with fixed basis set order $q$ as
\begin{align*}
 \myw_{q,u} ({X}_u) = \sum_{p=0}^q\myomega_{p,u}({X}_u),
\end{align*}
where the hierarchized ANOVA terms $\myomega_{p,u}$ are given by
\begin{align}\label{equ:w_pu_contribution}
  \myomega_{p,u} ({X}_u) := \myw_{p,u} ({X}_u)-\myw_{p-1,u}({X}_u),
\end{align}
with $\myw_{-1,u}=0$.
 We further
assume that the series
\begin{align}\label{equ:GammaSeries}
   V_\infty( {X}_1, \ldots, {X}_M) := \sum_{p\in\mathbb{N}}\sum_{ u\subseteq \{1, \ldots, M\}} \myomega_{p,u} ({X}_u)
\end{align}
is point-wise absolute convergent. Then, according to definition
(\ref{equ:SetVM}), it holds $V_\infty=V^{EMS}_\infty$. An example of such a decomposition which
corresponds to a linear molecule of length three is given in Appendix
\ref{sec:LinMolExample}.

Analogously, we can define hierarchical surpluses
\begin{align}\label{equ:w_pk_contribution}
  \myDelta_{p,k} (\Nvec{X}) := \myD_{p,k} (\Nvec{X})-\myD_{p-1,k}(\Nvec{X}),
\end{align}
with $\myD_{p,k}$ from equation (\ref{equ:BOSSANOVAENERGY-W}) for
$p\in\mathbb{N}$ and $\myD_{-1,k}=0$ and obtain an alternative
hierarchical series representation of the complete basis set limit in
the form
\begin{align}\label{equ:GammaSeries_k}
  V_\infty( {X}_1, \ldots, {X}_M) = \sum_{p\in\mathbb{N}} \sum_{0\leq
    k \leq M} \myDelta_{p,k} (\Nvec{X}).
\end{align}
Here, (\ref{equ:GammaSeries}) is a multilevel expansion with
contributions associated to the indices $(p,u) \in \mathbb{N}\times
\mathcal{P}(\{1, \dots, M\})$, while expansion
(\ref{equ:GammaSeries_k}) is a multilevel expansion with contributions
associated to the indices $(p,k) \in \mathbb{N}\times \{0, \dots,
M\}$.

\subsubsection{General multilevel approximation}
Now, we can truncate the infinite sum (\ref{equ:GammaSeries_k}) that
represents the complete basis set limit using a properly chosen
downward-closed\footnote{A set of indices ${I} \subseteq
  \mathbb{N}\times \{0, \ldots, M\}$ is called downward-closed or
  admissible, if for all $(p,k) \in {I}$ it follows $(\max\{p-1,
  0\},k) \in {I}$ and $(p,\max\{k-1,0\}) \in {I}$, see e.g.\
  \cite{gerstner2003dimension}.} index set
${I}\subseteq\mathbb{N}\times \{1 \ldots, M\}$ and obtain
\begin{align}\label{equ:VLambda}
  V_{I}( \Nvec{X}) :=\sum_{(p,k)\in{I}} \myDelta_{p,k} (\Nvec{X})
\end{align}
as an approximation to $V_\infty(\Nvec{X})$. Here, in the most simple case, the index set $I$ can be chosen as
\begin{align}
  \label{equ:ITDLalpha}
  I^{TD}_{L,a} := \left\{ (p,k)\,:\,a p+k \leq L \right\},
\end{align}
i.e.\ as an $\ell_1$-ellipse with scaling parameter $L$ which is related
to the desired accuracy and where $a$ prescribes a certain
weighting between $p$ and $k$.

Alternatively, a more general approximation than (\ref{equ:VLambda})
can be obtained by truncation of the infinite hierarchical series
(\ref{equ:GammaSeries}) using a properly chosen
downward-closed\footnote{A set of indices $\mathcal{I} \subseteq
  \mathbb{N}\times \mathcal{P}(\{1, \ldots, M\})$ is called
  downward-closed or admissible, if for all $(p,u) \in
  \mathcal{I}$ it follows $(\max\{p-1, 0\},u) \in \mathcal{I}$ and
  $(p,v) \in \mathcal{I}$ for all $v\subsetneq u$.} index set $\mathcal{I}
\subseteq \mathbb{N}\times \mathcal{P}(\{1, \ldots, M\})$, i.e.\
\begin{align}\label{equ:VLambda_u}
  V_{\mathcal{I}}( \Nvec{X})
  :=\sum_{(p,u)\in{\mathcal{I}}} \myomega_{p,u} (\Nvec{X}).
\end{align}
Note that the concept of downward-closedness is applicable here, since
the power set is a partially ordered
set. Note furthermore that every approximation (\ref{equ:VLambda}) can
be expressed in form of approximation (\ref{equ:VLambda_u}), but this
does not hold vice versa.

Now, let us shortly discuss the cost involved in the approximations
(\ref{equ:VLambda}) and (\ref{equ:VLambda_u}), respectively.  To this
end, we assume that the local energies $\widehat{V}_{p,u}({X}_{u})$
are computed with a cost of
\begin{align}
  \label{equ:localcostpu}
  c(p,u).
\end{align}  Then,
it is clear that the cost of computing a term $\myomega_{p,u}$ is
$\mathcal{O}(c(p,u))$, because all values $\widehat{V}_{p,u}({X}_{u})$
for $v \subsetneq u$ have been computed before, due to the
downward-closedness of $\mathcal{I}$. Therefore, the total cost of
employing the multilevel BOSSANOVA method (\ref{equ:VLambda_u}) with a
predefined index-set $\mathcal{I}\subseteq \mathbb{N}\times
\mathcal{P}(\{1, \ldots, M\})$ is given by
\begin{align*}
  Cost(\mathcal{I}) = \sum_{(p,u)\in\mathcal{I}} c(p,u).
\end{align*}
Analogously, in case of (\ref{equ:VLambda}) with $I\subseteq
\mathbb{N}\times \{0, \ldots, M\}$ the total cost read as
\begin{align*}
  Cost(I) = \sum_{(p,k)\in I} C(p,k),
\end{align*}
where
\begin{align}
  \label{equ:LocalCost-pk}
  C(p,k) := \sum_{|u|=k}c(p,u).
\end{align}
The advantage of the multilevel BOSSANOVA method stems from neglecting
the energy contributions and thus the costs of the less relevant subproblems.

\subsubsection{Quasi-optimal approximations}
\label{subsucsec:MLBN}
Let us now discuss a proper choice of a finite index sets with respect
to an appropriate benefit-cost setting similar to the sparse grid
construction \cite{bungartz2004sparse}.  Here, for reasons of
simplicity, we will discuss in detail the case of a finite index set
${I}$ in expansion (\ref{equ:VLambda}) only.  The case of a set
$\mathcal{I}$ works analogously. To this end, let us assume that it
holds $V_\infty=V_\infty^{EMS}$ and that
\begin{align}
  \label{equ:Bpk}
  |\mathcal{E}[\myDelta_{p,k}]| \lesssim B(p,k),
\end{align}
where $B:\mathbb{N}\times \{0, \ldots, M\} \to \mathbb{R}_+$ denotes
an upper bound of $|\mathcal{E}[\myDelta_{p,k}]|$, which we call the {\em
  local benefit}.  Then, we obtain the upper estimate
\begin{align}\label{equ:MBO-Total-Error}
  |\mathcal{E}[V_\infty^{EMS}-V_{I}]| = |\mathcal{E}[\sum_{(p,k)\in\mathbb{N}\times \{0, \ldots, M\}\setminus{I}}\myDelta_{p,k}]| \lesssim\sum_{(p,k)\in\mathbb{N}\times \{0, \ldots, M\}\setminus{I}} B(p,k).
\end{align}
Let us further assume that the cost contribution associated with a
hierarchical surplus $\myDelta_{p,k}$ can be estimated from above by an
appropriate {\em local cost} function $C:\mathbb{N}\times \{1, \ldots, M\}
\to \mathbb{R}_+$, compare (\ref{equ:LocalCost-pk}).

Now, we assume that we are allowed to spend just a limited total
specific cost. Then, the goal is to determine an index set (with at
most this associated total cost) such that the overall benefit is
maximized. The corresponding associated binary knapsack problem
consists in the determination of an index set ${I}_{\text{opt}}\subset
\mathbb{N}\times \{0, \ldots, M\}$ that maximizes the total benefit
under the constraint of a maximally allowed total cost $C_{max}$,
i.e.\
\begin{align}
  \label{equ:KnapsackBC}
  \max_{{I}}\sum_{(p,k)\in{I}}B(p,k)\quad \text{with} \quad  \sum_{(p,k)\in{I}}C(p,k) \leq C_{\text{max}}.
\end{align}
Its solution can be reduced to the discussion of the local
benefit-cost ratios $\frac{B(p,k)}{C(p,k)}$. Those indices $(p,k)$
with the largest benefit-cost ratios are taken into account first.
This is similar in spirit to the best $n$-term approximation. In the
framework of Sobolev spaces of dominating mixed smoothness such a
construction leads e.g.\ to quasi-optimal sparse grids. For a more
detailed discussion see e.g.\
\cite{bungartz2004sparse,gerstner2003dimension,nobile2015convergence}.
Note that the described construction is called {\em
  quasi-optimal}~\cite{nobile2015convergence}, since, amongst other
issues, the estimate of the total error (\ref{equ:MBO-Total-Error})
involves the triangle inequality\ only but no norm-equivalency.

Analogously to (\ref{equ:KnapsackBC}), the problem of the
determination of an index set $\mathcal{I}_{\text{opt}}\subset
\mathbb{N}\times \mathcal{P}(\{1, \ldots, M\})$ that maximizes the
total benefit $\sum_{(p,u)\in{\mathcal{I}_{\text{opt}}}}b(p,u)$
under the constraint of maximal allowed total cost
$\sum_{(p,u)\in{\mathcal{I}_{\text{opt}}}}c(p,u) \leq C_{\text{max}}$ reduces to
the discussion of local benefit-cost ratios $\frac{b(p,u)}{c(p,u)}$.

\subsubsection{Local cost model}
\label{sec:LocalCost}

In electronic structure calculations with standard correlation
consistent cc-pV$n$Z basis sets, the number of one-electron basis
functions per atom scales as
\cite{helgaker1997basis,halkier1998basis,truhlar1998basis}
\begin{eqnarray}
\label{nofbasefunc}
N_{\text{cc-pV$n$Z}} = \tfrac{1}{3}(n+1)(n+\tfrac{3}{2})(n+2).
\end{eqnarray}
Roughly speaking, the number of basis functions in the cc-pV$n$Z basis sets
grows with third order, i.e.\ $n^3$. In this article, we consider a
hierarchy of models which results from the basis sets cc-pV$n$Z
with $n\in\{3, 4, 5, 6, \ldots\}$ of Dunning et
al. \cite{dunning1989gaussian,wilson1996gaussian}. Here, cc-pVTZ
corresponds to the energy level $p=0$, cc-pVQZ corresponds to the
energy level $p=1$, cc-pV5Z corresponds to the energy level $p=2$ and
cc-pV6Z corresponds to the energy level $p=3$. Higher values of $p$
correspond to better quality.

Furthermore, we assume that the cost of the applied electronic
structure method, e.g.\ HF and DFT, scales with third order in the
number of one-electron basis functions.  Hence, we
estimate the cost of a single calculation for an $M$-atomic molecule
by
\begin{align*}
  Cost(M,n) \lesssim M^3n^{9},
\end{align*}
in case of HF and DFT, respectively.  Note that there are several
energy evaluations involved in the computation of a hierarchical
surplus $\myomega_{p,u}$. However, as already noted in the previous sections,
we assume that the computations for all backward neighbors in
definition (\ref{equ:w_pu_contribution}) have already been performed
and their results have been stored by a simple book-keeping procedure once
and for all. Thus, we assume that the local cost according to
(\ref{equ:localcostpu}) can be estimated by
\begin{align}
  \label{equ:LocalCosts}
  c(p,u) \leq \beta |u|^3(p+3)^{9},
\end{align}
where $\beta$ denote method dependent constant. Note that we neglect
here the additional hydrogen atoms resulting from the saturation
scheme.

\subsection{Adaptive multilevel BOSSANOVA}
\label{subsec:AdaptMBN}
To construct a quasi-optimal approximation as in
Section~\ref{subsucsec:MLBN}, a priori local cost and benefit
estimators are needed. Let us now discuss the case when there is no
suitable a priori estimate for the local benefit available. Then, a
possible methodology are so-called dimension-adaptive approaches,
which are adaptive {\em greedy-type} algorithms using {\em a
  posteriori} estimates~\cite{gerstner2003dimension}. Under the
assumption that the benefit-cost ratios obey some kind of appropriate
decay, these type of algorithms try to find quasi-optimal index sets
in an iterative procedure.  To this end, as always for an adaptive
heuristics, two main ingredients are needed -- an error indicator and
a refinement rule.  Here we propose a hybrid a priori/a posteriori
approach, since we determine the local benefit a posteriori and the
local cost a priori.

\subsubsection{Adaptive index sets in $\mathbb{N}\times \{0, \ldots, M\}$}
First, we introduce Algorithm~\ref{alg:Dim-Adapt-pk} which builds up a
set of indices $(p,k)$ such that the infinite sum
(\ref{equ:GammaSeries_k}) is approximated by a quasi optimal
approximation (\ref{equ:VLambda}) up to the desired accuracy at
minimal cost.
To this end, we refine in each step of our approach the current index
set $I$ by adding those indices $(p,k)\in\mathbb{N}\times \{0, \ldots,
M\}\setminus I$ with a benefit-cost ratio greater or equal than an a
priori chosen factor $[1-\alpha]$ times the highest benefit-cost ratio
$\max_{(p,k)\in A} \eta(p,k)$. Here, it is not necessary to consider
all of the possible candidates $(p,k)\in\mathbb{N}\times \{0, \ldots,
M\}\setminus I$ (which would be too many), but in each iteration step
we only take those indices into account which are in the direct
neighborhood of the actual set $I_\alpha$ and which result in
downward-closed sets.  Algorithm~\ref{alg:Dim-Adapt-pk} uses the local
cost estimate (\ref{equ:LocalCost-pk}) and computes the local benefit
a posteriori according to (\ref{equ:Bpk}) by
\begin{align*}
  |\mathcal{E}[\myDelta_{p,k}]|.
\end{align*}
Thus, for an index $(p,k)\in \mathbb{N}\times \{0, \ldots, M\}$, the
associated benefit-cost ratio is given by
\begin{align}
  \label{equ:eta-pk}
  \eta(p,k) :=  \frac{|\mathcal{E}[\myDelta_{p,k}]|}{C(p,k)}.
\end{align}
Altogether, to a given factor $\alpha\in[0, 1]$ and a maximal cost $C_{max}$,
Algorithm~\ref{alg:Dim-Adapt-pk} tries to generate a quasi-optimal index set $I_\alpha$
and an associated approximation
\begin{align}
  \label{equ:V-dim-adapt-pk}
  V_{{I}_{\alpha}}( \Nvec{X}) := \sum_{(p,k)\in{I}_{{\alpha}}} \myDelta_{p,k} (\Nvec{X}).
\end{align}
\begin{algorithm}
  \caption{Adaptive algorithm to construct a quasi-optimal index
    sets ${I}_\alpha \subseteq \mathbb{N}\times\{0,\ldots, M\}$ for a molecule with interaction graph $G$ to a given maximal allowed cost $C_{max}$, where $\alpha \in [0\, 1]$ is an input parameter.}\label{alg:Dim-Adapt-pk}
  \begin{algorithmic}
    \STATE {\bf Initialize:} ${I}_\alpha:= \emptyset$, $cost := 0$, ${A}:= \{(0,1)\}$
    \WHILE{$cost < C_{max}$}
    \STATE
    \begin{enumerate}
    \item Compute the values  $\eta(p,k):=|\mathcal{E}[\myDelta_{p,k}]|/C(p,k)$ for all $(p,k)\in A$ for which $\eta(p,k)$ has not been computed yet and set $cost := cost + C(p,k)$ respectively.
    \item Select new indices $\{ (p,k) \in A \,:\,  \eta(p,k) \geq (1-\alpha)  \max_{(q,l)\in A} \eta(q,l)\}$ and move them from $A$ to $I$.
    \item Generate new admissible active index set $A$ for $I_\alpha$, i.e.
    $A := \{ (p,k) \in \mathbb{N}\times\{0,\ldots, M\}\setminus {I}_\alpha \, : \, (p,k) \cup {I}_\alpha \text{ is downward-closed}\}$
    \end{enumerate}
    \ENDWHILE
  \end{algorithmic}
\end{algorithm}

\subsubsection{Adaptive index sets in $\mathbb{N}\times \mathcal{P}(\{1, \ldots, M\})$}
\label{subsubsec:AMBNpu}
An alternative algorithm, which provides the application of different
basis sets at different local parts of a molecule is as follows:
Instead to rely on set of indices $I \subseteq \mathbb{N}\times \{0,
\ldots, M\}$ as before, we now construct an index-set $\mathcal{I}
\subseteq \mathbb{N}\times \mathcal{P}(\{1, \ldots, M\})$, considering
the whole power set $\mathcal{P}(\{1, \ldots, M\})$ of $\{1, \ldots,
M\}$. This allows a much better adaption to the specific molecule in
consideration.  This straightforward generalization of
Algorithm~\ref{alg:Dim-Adapt-pk} is presented in
Algorithm~\ref{alg:Dim-Adapt-pu}.

In contrast to
Algorithm~\ref{alg:Dim-Adapt-pk}, Algorithm~\ref{alg:Dim-Adapt-pu} is
based on the expansion (\ref{equ:GammaSeries}) and uses benefit-cost
ratios associated with indices $(p,u) \in \mathbb{N}\times
\mathcal{P}(\{1, \ldots, M\})$ defined by
\begin{align}
  \label{equ:theta-pu}
  \theta(p,u):=\frac{b(p,u)}{c(p,u)},
\end{align}
where the local benefit is defined by
\begin{align*}
  b(p,u):=|\mathcal{E}[\myomega_{p,u}]|
\end{align*}
and the local cost is given according to (\ref{equ:localcostpu}).
Analogously to Algorithm~\ref{alg:Dim-Adapt-pk},
Algorithm~\ref{alg:Dim-Adapt-pu} tries to generate a
quasi-optimal index set $\mathcal{I}_\alpha$ where the corresponding
approximation is given by
\begin{align}
  \label{equ:V-dim-adapt-pu}
  V_{\mathcal{I}_{\alpha}}( \Nvec{X}) := \sum_{(p,u)\in\mathcal{I}_{{\varepsilon_i}}} \myomega_{p,u} ({X}_u).
\end{align}
\begin{algorithm}
  \caption{Adaptive algorithm to construct a quasi-optimal index
    sets $\mathcal{I}_\alpha \subseteq \mathbb{N}\times\mathcal{P}({\{1, \ldots, M\}})$ for a molecule with interaction graph $G$ to a given maximal allowed cost $C_{max}$, where $\alpha \in [0\, 1]$ is an input parameter.}\label{alg:Dim-Adapt-pu}
  \begin{algorithmic}
    \STATE {\bf Initialize:} $\mathcal{I}_\alpha:= \emptyset$, $cost := 0$, $\mathcal{A}:= \{(0,1)\}$
    \WHILE{$cost < C_{max}$}
    \STATE
    \begin{enumerate}
    \item Compute the values $\theta(p,u):=|\mathcal{E}[\myomega_{p,u}]|/c(p,u)$ for all
      $(p,u)\in \mathcal{A}$ for which $\theta(p,u)$ has not been
      computed yet and set $cost := cost + c(p,u)$ respectively.
    \item Select new indices $\{ (p,u) \in \mathcal{A} \,:\,  \theta(p,u) \geq (1-\alpha)  \max_{(q,v)\in \mathcal{A}} \theta(q,v)\}$ and move them from $\mathcal{A}$ to $\mathcal{I}$.
    \item Generate new admissible active index set $\mathcal{A}$ for $\mathcal{I}_\alpha$, i.e.
    $\mathcal{A} := \{ (p,u) \in \mathbb{N}\times\mathcal{P}({\{1, \ldots, M\}})\setminus \mathcal{I}_\alpha \, : \, (p,u) \cup \mathcal{I}_\alpha \text{ is downward-closed}\}$
    \end{enumerate}
    \ENDWHILE
  \end{algorithmic}
\end{algorithm}

Let us finally remark that in contrast to the relation
(\ref{equ:LocalCost-pk}) for the local costs $C(p,k)$ and $c(p,k)$,
the local benefit $|\mathcal{E}[\myDelta_{p,k}]|$ is not equal to the associated sum of
local-benefits $\sum_{u\subseteq \{1, \ldots, M\}, |u|=k}b(p,u)$.

\subsubsection{Parallel cost model}
\label{subsec:Parallel}
Note here that using a simple book-keeping procedure, all involved
energy evaluations $\widehat{V}_{p,u} ({X}_{u})$ only have to be
computed once and for all and can in particular be performed in
parallel. In such a parallel cost model, the total cost of employing
the multilevel BOSSANOVA method with a predefined index-set ${I}\subset \mathbb{N}\times \{0, \ldots, M\}$ can
be estimated from above by
\begin{align*}
  C^{PA}({I}) := \max_{(p,k)\in{I}}\max_{u\subseteq \{1, \ldots, M\},|u|=k}c(p,u).
\end{align*}
Analogously, the total costs in case of a predefined index-set $\mathcal{I}\subset \mathbb{N}\times \mathcal{P}(\{1, \ldots, M\})$ can be
estimated from above by
\begin{align*}
  C^{PA}(\mathcal{I}) := \max_{(p,u)\in\mathcal{I}}c(p,u).
\end{align*}

In case of the adaptive Algorithm~\ref{alg:Dim-Adapt-pu} all benefits
and respective benefit-cost ratios can be computed in parallel in step
1 of the while loop. Therefore, the total cost of
Algorithm~\ref{alg:Dim-Adapt-pu} in the parallel cost model can be
computed by setting $cost := cost + \max_{(p,u)}c(p,u)$ in step 1 of
the while loop. In an analogous way Algorithm~\ref{alg:Dim-Adapt-pk} can also be
modified according to the parallel cost model.

\section{Numerical experiments}
\label{sec:NumRes}
Now we present the results of our numerical experiments. This section
is divided into two parts. In the first part, we perform a numerical
study on the decay properties of the hierarchical surpluses of several
molecules. In the second part, we will present and discuss the numerical results
corresponding to the different approximation approaches which were presented in
Section~\ref{sec:AdaptiveMultilevelBOSSANOVA}.

In all involved electronic structure calculations we apply the
Massively Parallel Quantum Chemistry (MPQC) Program
\cite{janssen2004massively}. In addition, we use the software package
{\em MoleCuilder - a molecular builder} \cite{molecuilder} for the
fragmentation and saturation process involved in the various multilevel BOSSANOVA
approaches. Moreover, in the numerical experiments applying the DFT method
we use the common B3LYP exchange-correlation functional.

\subsection{Numerical study of benefit-cost ratios}
\label{subsec:NumericalStudy}

In a first step we analyze the benefit-cost ratios $\eta(p,k)$,
involved in Algorithm \ref{alg:Dim-Adapt-pk}, which are defined
according to (\ref{equ:eta-pk}). We first study two small
molecules. Our results for heptane ($C_7H_{16}$ - a small linear
molecule) and acrylamide ($C_3H_5NO$ - a small molecule with branches)
applying the HF method within the
decomposition~(\ref{equ:GammaSeries_k}) are given in Figure
\ref{fig:eta-HF-heptane-acrylamide}.
\begin{figure}[h]
  \centering
  \begin{tabular}{cc}
    \includegraphics[width=0.45\textwidth]{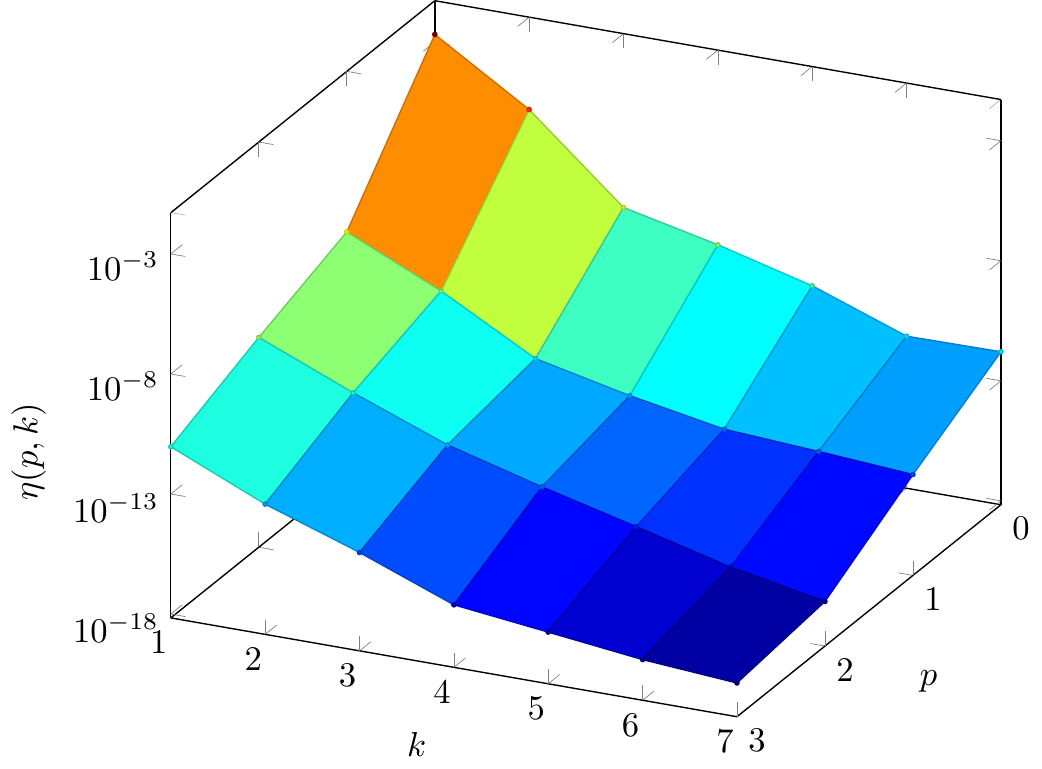}&
    \includegraphics[width=0.45\textwidth]{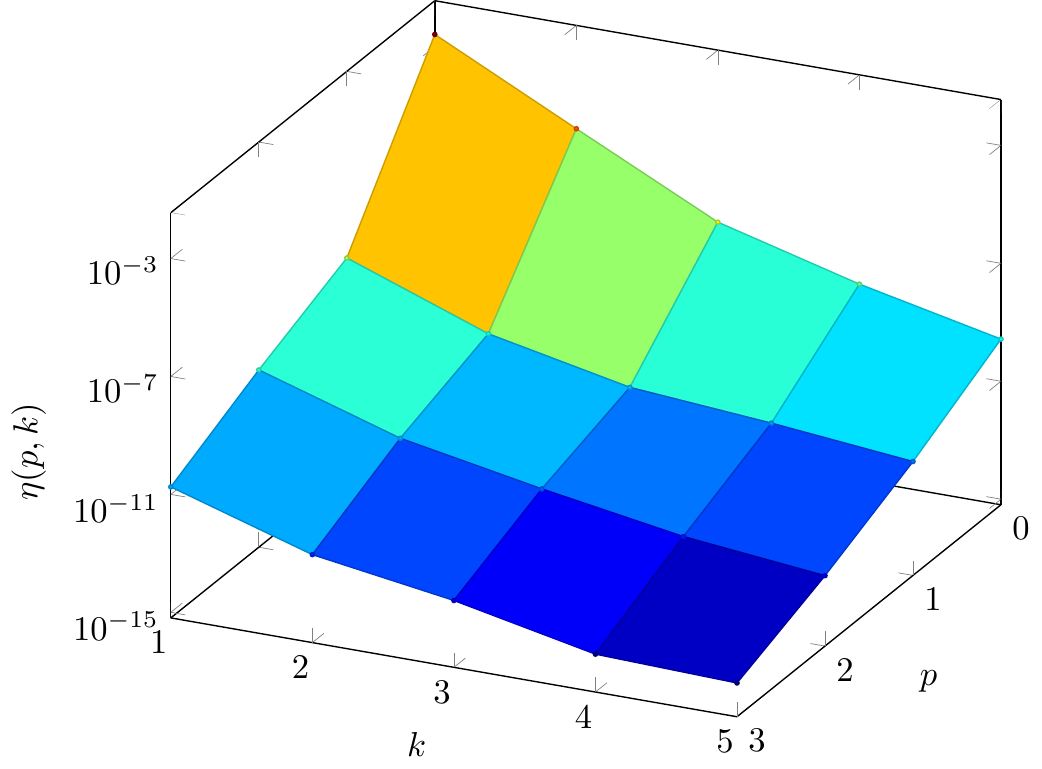}\\
  \end{tabular}
  \caption{Benefit-cost ratios $\eta(p,k)$ using the HF method corresponding to basis
    sets cc-pVTZ to cc-pV6Z, i.e. $p=0, \dots, 3$. {\em Left}:
    heptane ($C_7H_{16}$). {\em Right}: acrylamide ($C_3H_5NO$).}
\label{fig:eta-HF-heptane-acrylamide}
\end{figure}
We clearly see that the benefit-cost ratios $\eta(p,k)$ decay with $p+k$ and are in particular
approximately constant along the lines $l=p+k$ for a given fixed
$l$. This corresponds to the decay behavior in the regular sparse
grids case.

Moreover, for acrylamide, acetanilide ($C_8H_9NO$ - a small molecule
with a ring structure) and several alkanes ($C_5H_{12}$, $C_6H_{14}$,
$C_7H_{14}$, $C_{20}H_{42}$), we depict the benefit-ratio $\eta(p,k)$
in dependence of the sum $p+k$ in Figure
\ref{fig:eta-molecules-proteins}.
\begin{figure}[h]
  \centering
  \begin{tabular}{cc}
    \includegraphics[width=0.45\textwidth]{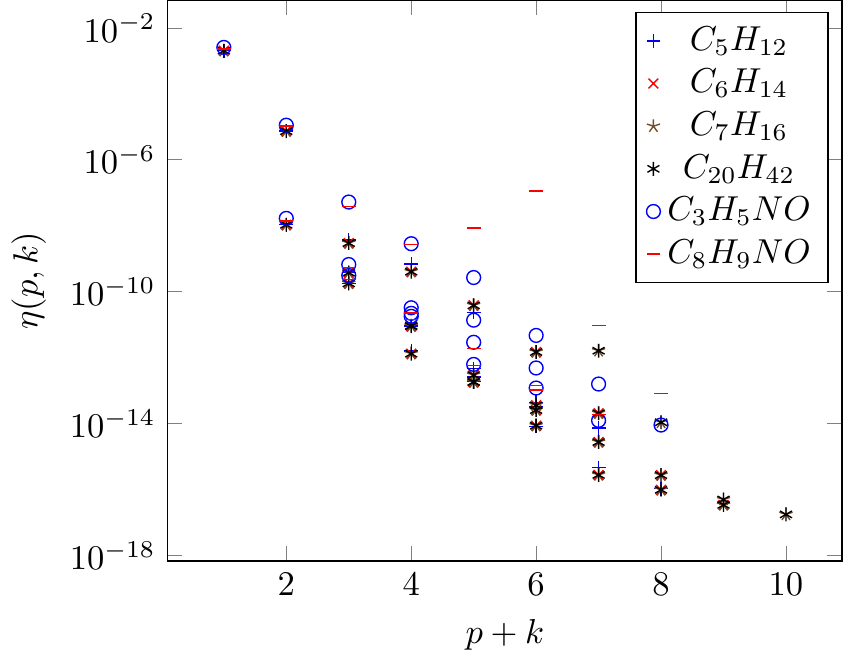}&
    \includegraphics[width=0.45\textwidth]{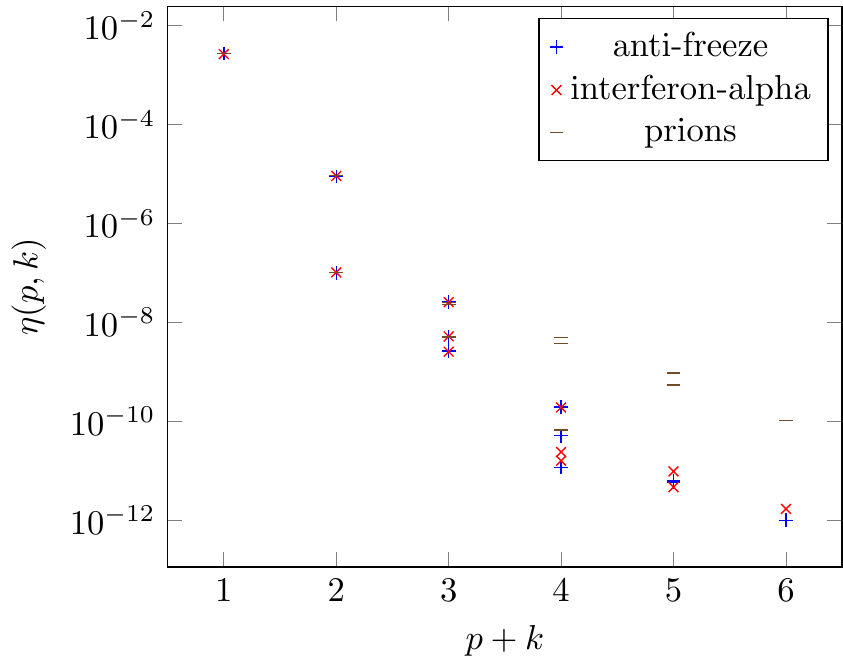}\\
  \end{tabular}
  \caption{Benefit-cost ratios $\eta(p,k)$  versus
    the sum $p+k$ for several alkane, acrylamide ($C_3H_5NO$),  acetanilide ($C_8H_9NO$) applying the HF method ({\em left}) and three proteins applying the DFT approach ({\em right}).}
\label{fig:eta-molecules-proteins}
\end{figure}
In addition, we show there the
benefit-cost ratio $\eta(p,k)$ for three large molecules, i.e.\ an
anti-freeze protein (992 atoms), an interferon-$\alpha$ protein (2698
atoms) and a prion (1688 atoms).  It can be seen from Figures
\ref{fig:eta-HF-heptane-acrylamide} and
\ref{fig:eta-molecules-proteins} that indeed our numerical results
show a decay of the benefit-cost ratio $\eta(p,k)$ with an increase of
the sum of the discretization parameters $p+k$, whereas acetanilide
exhibits some outliers probably due to its specific structure
including a ring. Also for the three large molecules, we observe
roughly a decay of $\eta(p,k)$ with $p+k$. Here, however, the
situation is much more less clear. The prion molecule seems to exhibit
a somewhat different behavior. Altogether this advocates the more
refined adaptive approaches of Algorithm \ref{alg:Dim-Adapt-pk} and
especially Algorithm \ref{alg:Dim-Adapt-pu}.

Therefore, in a second step, we study the benefit-cost ratios
$\theta(p,u)$ given in equation~(\ref{equ:theta-pu}) and used in
Algorithm~\ref{alg:Dim-Adapt-pu}.  Our numerical results given in
Figure~\ref{fig:theta-molecules-proteins} also exhibit a decay of the
benefit-cost ratio $\theta$ with an increase of the sum
$p+|u|$. However, compared to the range of the values of the
benefit-ratios $\eta(p,k)$ with $p+k=l$, we observe a much larger
range of deviation of the values of the benefit-cost ratios
$\theta(p,u)$ with $p+|u|=l$ in most cases, which holds in particular
for acetanilide and the proteins for e.g.\ $l\in\{3,4,5,6\}$.  This
suggests that for simple chain-like molecules the indicators
$\eta(p,k)$ and thus Algorithm~\ref{alg:Dim-Adapt-pk} are probably
sufficient. For more complex molecules however the indicators
$\theta(p,u)$ and thus Algorithm~\ref{alg:Dim-Adapt-pu} promise more
refined results and indeed might be superior.
\begin{figure}[h]
  \centering
  \begin{tabular}{cc}
    \includegraphics[width=0.45\textwidth]{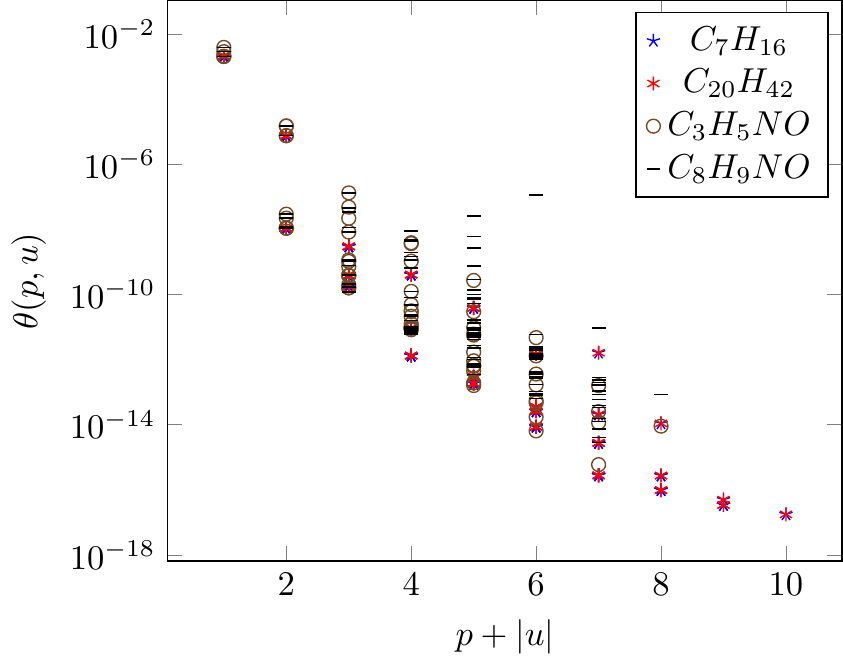}&
    \includegraphics[width=0.45\textwidth]{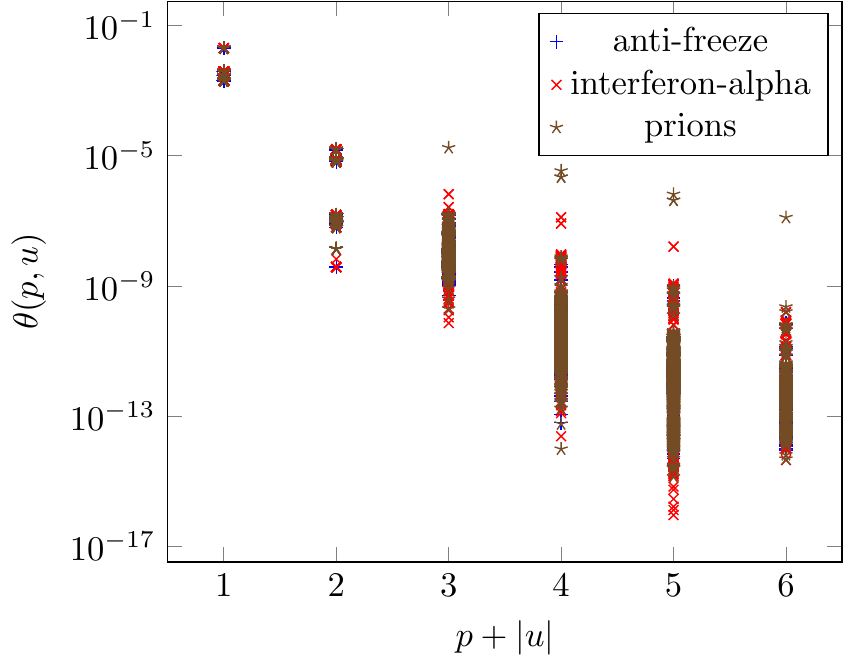}\\
  \end{tabular}
  \caption{Benefit-cost ratios $\theta(p,u)$ versus
    the sum $p+|u|$ for several alkane, acrylamide, acetanilide  and three proteins.}
\label{fig:theta-molecules-proteins}
\end{figure}

\subsection{Numerical results for the approximation of energy}
Now we consider the approximation of the energy of heptane by the
approaches introduced in
Section~\ref{sec:AdaptiveMultilevelBOSSANOVA}. To this end, we will
shortly recall our notation of the energy approximations associated to
the considered methods: $V_p^{HF}$ denotes the energy of the
conventional HF method using a cc-pV$(n)$Z basis set with $n=p+3$ and
$V_{p,k}^{BN}$ denotes the energy of the BOSSANOVA approach with basis
set level $p$ and bond order $k$ according to equation
(\ref{equ:BOSSANOVAENERGY}).  The energy of the multilevel BOSSANOVA
approach for a general admissible index set $I$ is defined
in~(\ref{equ:VLambda}) and is denoted by $V_{I}$. Here, inspired by
the observations in Section~\ref{subsec:NumericalStudy}, we first restrict
ourselves to use the index sets $I=I^{TD}_{L,1}$ as defined
in~(\ref{equ:ITDLalpha}), which correspond to the regular sparse grids
approach.
Furthermore we consider the energy approximations $V_{I_{\alpha}}$ and
$V_{\mathcal{I}_{\alpha}}$ according to the adaptive
Algorithm~\ref{alg:Dim-Adapt-pk} with equation
(\ref{equ:V-dim-adapt-pk}) and the adaptive Algorithm~\ref{alg:Dim-Adapt-pu} with
equation (\ref{equ:V-dim-adapt-pu}), respectively.
\begin{figure}[h!]
  \centering
  \begin{tabular}{c}
    \includegraphics[width=0.55\textwidth]{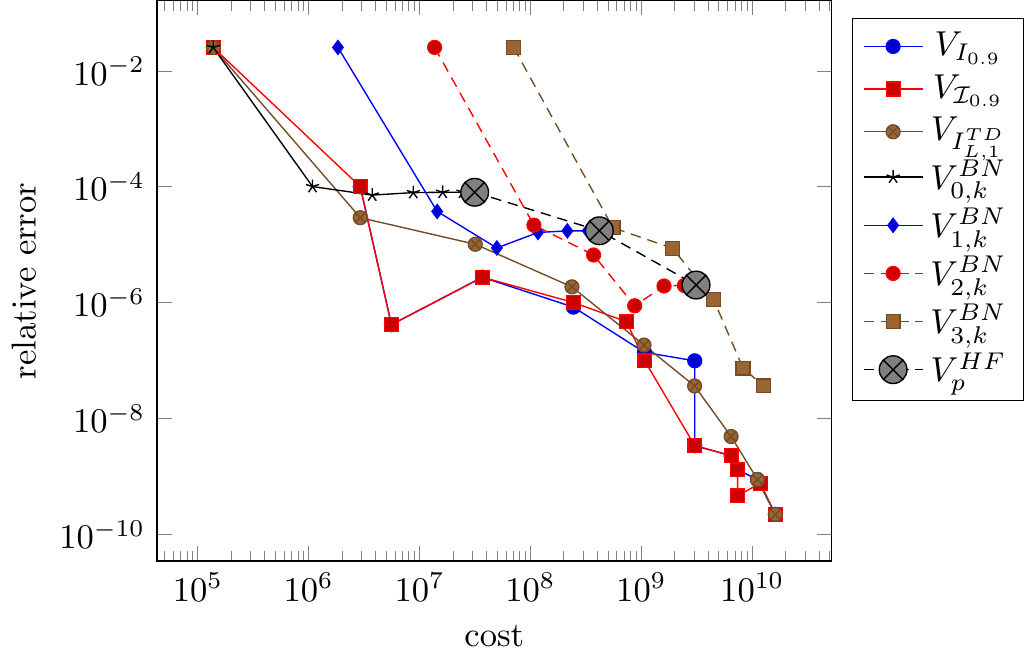}
  \end{tabular}
  \caption{Relative error versus costs of different approximation
    methods for heptane, where we choose $V_4^{HF}$ as reference
    energy, i.e.\ the HF energy computed with the help of the cc-pV6Z
    basis set.  Shown are for $p=0,1,2$ the energies $V_p^{HF}$ of the
    HF method, for $k=1,\ldots, 6$ the energies $V_{p,k}^{BN}$ of the
    BOSSANOVA approach and for $L=1,\ldots,9$  the energies
    $V_{I^{TD}_{L,1}}$ of the the multilevel BOSSANOVA method.
    For the adaptive multilevel
    BOSSANOVA approach the energy of each step of
    Algorithm~\ref{alg:Dim-Adapt-pk} and
    Algorithm~\ref{alg:Dim-Adapt-pu} for $V_{{I}_{0.9}}$ and
    $V_{\mathcal{I}_{0.9}}$ are shown, respectively.}
\label{fig:approx-HF-heptane}
\end{figure}

For the heptane example we choose the HF approach as the local
electronic structure method.  The results are given in
Figure~\ref{fig:approx-HF-heptane}. We observe that the three
multilevel BOSSANOVA approaches, i.e.\ $V_{I^{TD}_{L,1}}$,
$V_{{I}_{0.9}}$ and $V_{\mathcal{I}_{0.9}}$, exhibit a substantially
better convergence behavior than the HF and the standard BOSSANOVA
method, i.e.\ than $V_p^{HF}$ and $V_{p,k}^{BN}$, respectively.
Furthermore, we see that the adaptive variants, i.e.\ $V_{{I}_{0.9}}$
and $V_{\mathcal{I}_{0.9}}$, give almost the same results (for values
with a relative error larger than 1.e-5 they are even exactly the
same) and that both indeed show even a slightly better convergence
behavior than the variant with a priori chosen index sets, i.e.\
$V_{I^{TD}_{L,1}}$.
\begin{table}[h]
  \caption{Comparison of the costs of the different methods
    to obtain a relative error less than a value of 1.e-5 for heptane. In addition the speed up factor for each method compared to the cost of the conventional HF method is given.}
  \label{tab:CompareTable1}
  \centering
  \begin{tabular}{c|ccccc}
     & rel. err. & cost & speed-up  & parallel cost & parallel speed-up\\\hline
    $V_2^{HF}$          & {2.02e-6} &   {6.7e+08}& {1.0}  & 6.7e+8 & {1.0} \\
    $V^{BN}_{2,3}$       & {6.67e-6} &   {3.7e+08}& {1.8}  & 5.3e+7 & {12.7}\\
    $V_{I^{TD}_{4,1}}$    & {1.88e-6} &   {2.4e+08}& {2.8}& 1.6e+7 & {42.9}\\
    $V_{{I}_{0.9}}$       & {2.75e-6} &   {3.7e+07}& {18.1}& 2.9e+6 & {230.2}\\
    $V_{\mathcal{I}_{0.9}}$ & {2.75e-6} &   {3.7e+07}& {18.1}& 2.9e+6 & {230.2}\\ \hline
  \end{tabular}
\end{table}
In addition, we compare in Table~\ref{tab:CompareTable1} the costs of
all methods to obtain a specific accuracy. We obtain a
speed-up factor of about $18$ for the cost of the two adaptive
multilevel BOSSANOVA energies, i.e.\ $V_{{I}_{0.9}}$ and
$V_{\mathcal{I}_{0.9}}$, compared to the energy value computed by the
conventional HF method, i.e.\ $V_2^{HF}$. For the multilevel BOSSANOVA
energy $V_{I^{TD}_{4,1}}$ we still get a speed up factor of about
$2.8$ compared to the cost of $V_2^{HF}$. In the parallel cost model
we obtain even a speed-up factor of about $230$ for $V_{{I}_{0.9}}$
and of about $43$ for $V_{I^{TD}_{4,1}}$. We thus see that, already
for a small molecule like heptane, the gain of our new method is
substantial. For larger molecules it should be even more profound.

Therefore, we apply our new adaptive approaches to the more complex
molecules acetanilide ($C_8H_9NO$) and acrylamide
($C_3H_5NO$). Moreover, we varied the parameter $\alpha$ in a wider
range of values. It turned out that values greater than $\alpha=0.75$
perform quite well. Thus, we considered $\alpha=0.9$ and $\alpha=0.99$
in the following. Our numerical results, given in
Figure~\ref{fig:approx-HF-complex}, show that the
convergence behavior of both algorithms is similar.  However, the
results for acetanilide suggest that for more complex molecules the
more general dimension-adaptive approximation approach according to
(\ref{equ:V-dim-adapt-pu}) and
Algorithm~\ref{alg:Dim-Adapt-pu} is potentially superior compared to
the more restrictive dimension-adaptive technique corresponding to
Algorithm~\ref{alg:Dim-Adapt-pk} and
(\ref{equ:V-dim-adapt-pk}).  Note that each approximation in form of
equation (\ref{equ:V-dim-adapt-pk}) can be represented also in the
form corresponding to (\ref{equ:V-dim-adapt-pu}), but this
does not hold vice versa.
\begin{figure}[h!]
  \centering
  \begin{tabular}{cc}
    \includegraphics[width=0.45\textwidth]{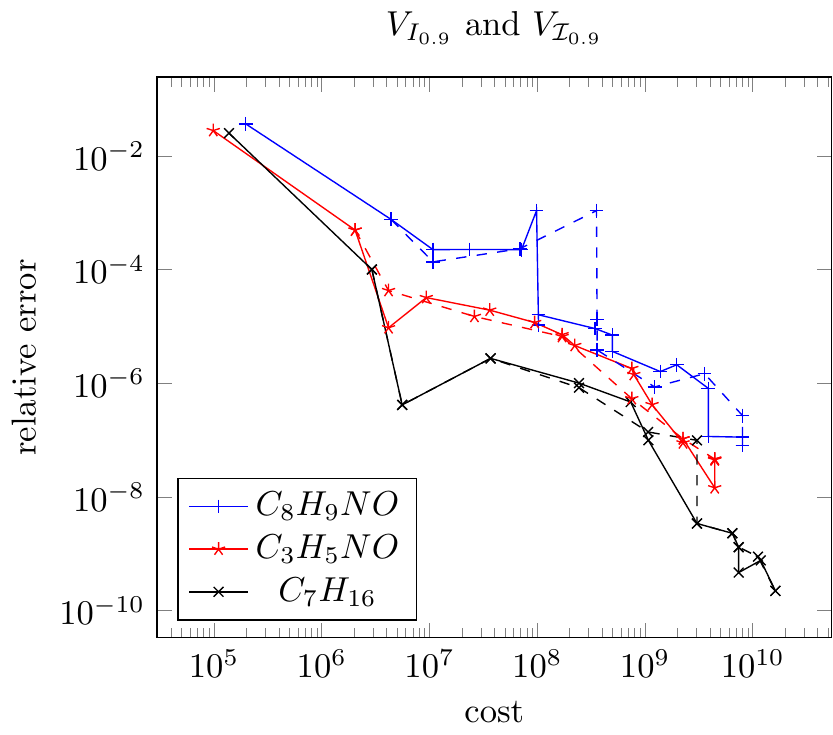}&
    \includegraphics[width=0.45\textwidth]{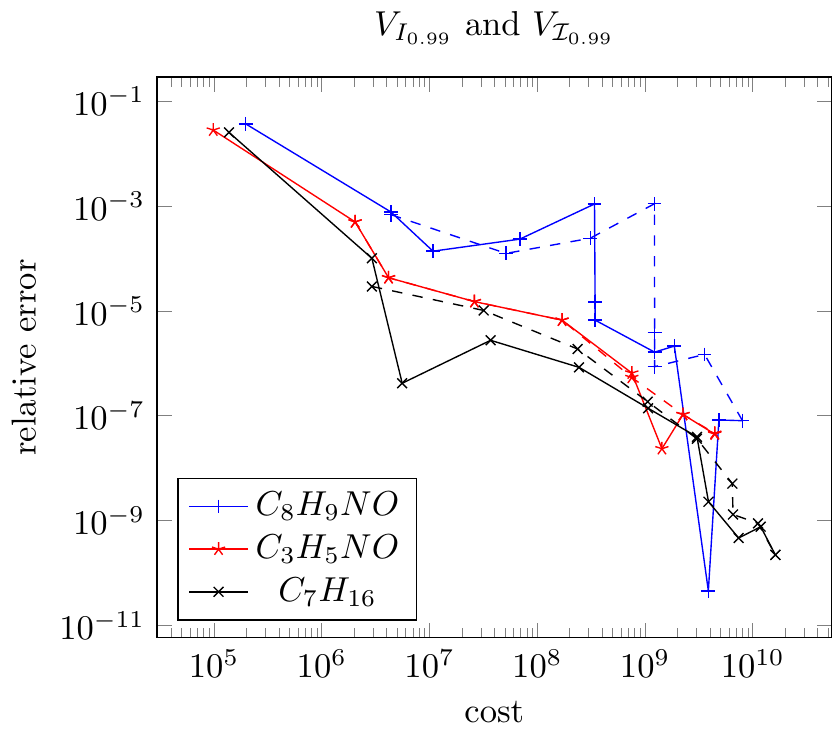}
  \end{tabular}
  \caption{Relative error versus costs of Algorithm~\ref{alg:Dim-Adapt-pk} and Algorithm~\ref{alg:Dim-Adapt-pu}. The solid lines correspond to the errors of the energies $V_{\mathcal{I}_{\alpha}}$ of Algorithm~\ref{alg:Dim-Adapt-pu}, while the dashed lines  correspond to the errors of the energies $V_{{I}_{\alpha}}$ of Algorithm~\ref{alg:Dim-Adapt-pk}.}
\label{fig:approx-HF-complex}
\end{figure}

The larger the molecules are, the more profound the gain for our
BOSSANOVA approach is compared to the conventional methods applied to
the overall molecule. To show this we finally apply our new adaptive
multilevel method to molecules which consist of so many atoms such
that the conventional methods can no longer be applied in
practice. To this end, as in Section~\ref{subsec:NumericalStudy}, we
will consider the anti-freeze protein, the interferon-$\alpha$ protein
and the prion, where we now use the DFT method within our multilevel
BOSSANOVA approach.
\begin{figure}[h!]
  \centering
  \begin{tabular}{c}
    \includegraphics[width=0.45\textwidth]{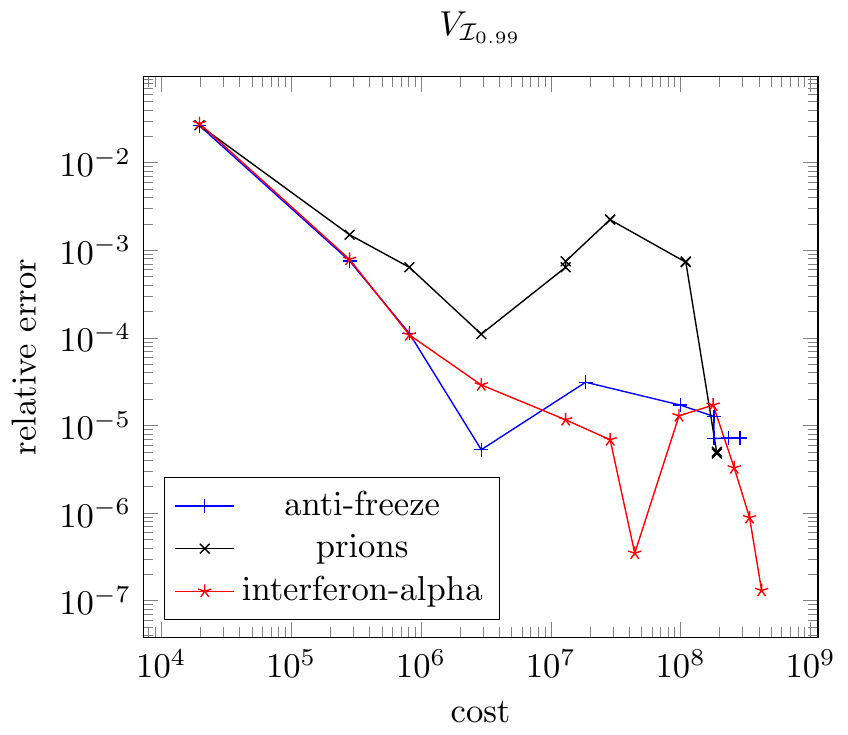}
  \end{tabular}
  \caption{Relative error versus parallel costs of Algorithm~\ref{alg:Dim-Adapt-pu}, where we choose $V_{3,3}^{BN}$ as reference energy, i.e.\ the energy of the BOSSANOVA approach for basis set cc-pV6Z and bond order three.
  }
\label{fig:approx-DFT-proteins}
\end{figure}
The numerical results are given in
Figure~\ref{fig:approx-DFT-proteins}. There we choose as the reference
energy to compute the relative error the BOSSANOVA energy
$V_{3,3}^{BN}$, apply the parallel cost model and choose~$\alpha=0.99$.

Altogether, the results of our numerical experiments suggest that
indeed the adaptive multilevel BOSSANOVA approaches can be efficiently
and successfully applied to larger molecules. Especially
Algorithm~\ref{alg:Dim-Adapt-pu} promises, due to its local refinement
property, to enable systematic quasi-optimal approximations in a
cost-efficient manner.

\section{Concluding Remarks}\label{sec:Con}
In this article we presented the adaptive multilevel BOSSANOVA
decomposition approach for the approximate ground state solution to the electronic
Schr\"odinger equation for a given molecular system.  Here, we followed
the idea of sparse grids and the adaptive combination technique to
obtain systematically quasi-optimal approximations, i.e.\ a specific truncation of
the hierarchical series such that the total benefit is maximized for a
fixed amount of costs.

We described and discussed an a priori truncation scheme and in
particular two a posteriori dimension-adaptive algorithms for a
seamless coupling of local computations whith a high level basis set
where needed and low level basis sets where locally sufficient.
We gave numerical results for small chain molecules, where we obtained
substantial speed up factors compared to the cost of the conventional
HF method. We furthermore presented numerical results for three large
proteins.  Here, the more general variant of the dimension-adaptive
approach, which allows for local adaptivity, seems to be superior
compared to the more restrictive variant, which does not allow for
local adaptivity. Let us point out that our approach is trivial
to parallelize since the evaluation of each fragment by an appropriate
solver can be done independently, see e.g.\ \cite{Heber2014}.

In this article, for reasons of simplicity, we did not investigate the
treatment of aromatic systems with delocalized electrons with ring
structures in more detail. For such problems, there is surely room for
further improvement of our algorithms. A simple modification of the
conventional BOSSANOVA approach was already successfully applied to such systems in
\cite{griebel2014bond,Heber2014}.
The impact of the neglected long-range Coulomb energy on the
accuracy of the method and techniques to recover this contribution were given
elsewhere, see \cite{Heber2014}.

Note furthermore that the BOSSANOVA
approach is not free of empirical parameters due to the necessity to
saturate dangling bonds with hydrogen in the fragmentation
process. But the typical bond lengths and angles of hydrogenated
systems are well assessed by measurements.

\section{Acknowledgments}
The authors gratefully acknowledge the Hausdorff Center for
Mathematics (HCM) and the Cooperative Research Centre (SFB 1060) at
University of Bonn for the financial support. The authors also thank
Frederik Heber (Numerical Analysis and Applied Mathematics, Saarland
University) for useful discussions.

\section*{Appendix}
\begin{appendix}
  \section{Modified energy terms $\widehat{V}^{G}_{p,u}$}\label{sec:AppendixGraph}
  Let $G$ be an connected interaction graph of a molecule. Then,
  according to the BOSSANOVA approach, the modified energy associated
  with a fragment $u$ is given by
  \begin{align*}
    \widehat{V}^{G}_{p,\{i_1, \ldots, i_{k}\}}({X}_{\{i_1, \ldots, i_{k}\}}) := \sum_{u \in \mathcal{C}(G)}
    V_p\left({X}_{i_1}, \ldots, {X}_{i_k}, R^G_1(u), \ldots, R^G_{h^G(u)}(u)\right),
  \end{align*}
  where
  \begin{multline*}
    \mathcal{C}((P,E)) := \bigl\{u \subseteq P \,:\, \text{ $(P_u,
      E_u)\subseteq (P,E)$ is connected and}\\\text{for all $v \supsetneq u$
      subgraph $(P_v, E_v)$ is not connected}\bigr\}.
  \end{multline*}
  Thus, it holds the relation
  \begin{align*}
    \widehat{V}_{p,A\cup B}^G(X_{A\cup B}) =  \widehat{V}_{p,A}^G(X_{A}) +  \widehat{V}_{p,A}^G(X_{A}),
  \end{align*}
  for all $p\in\mathbb{N}$ and for all pairs of subsets $A, B
  \subseteq\{1, \ldots, M\}$ with disconnected induced subgraphs $G_A,
  G_B \subseteq G$. Moreover, we can derive that the corresponding
  hierarchical surplus $\myw_{p,A\cup B}$ vanishes:
  \begin{lemma}\label{lem:Graph}
    Let $G=({P},{E})$ be an interaction graph. Let $A, B
    \subseteq {P}$, $A \cap B = \emptyset$ and let the
    subgraphs $G_A$ and $G_B$ induced by $A$ and $B$, respectively, be disconnected.
    Then
    \begin{equation*}
      W^G_{p,A\cup B}({X}_{A\cup B}) = 0.
    \end{equation*}
  \end{lemma}
  \begin{proof}
    We use induction: The base case can be easily seen for graphs
    $G=({P},{E})$ with sets $|{P}| \leq 2$.  Let
    us assume that the statement holds for graphs
    $G=({P}^\prime,{E}^\prime)$ with
    $|{P}^\prime|\leq n$. Now let $G=({P},{E})$
    with $|{P}|=n+1$.
    Note that from the recursive definition of $\widehat{V}^{G}_{p,a}$ it immediately follows that
    \begin{equation*}
      \widehat{V}^{G}_{p,a}({X}_a) = \sum_{b \subseteq a} \myw_{p,b}({X}_b)
    \end{equation*}
    holds for all $a \subseteq {P}$.  With
    $\widehat{V}^{G}_{p,A\cup B}(X_{A\cup B}) =  \widehat{V}^{G}_{p,A}(X_{A}) +  \widehat{V}^{G}_{p,B}(X_{B})$
    and $W^{G}_{p,\emptyset}=0$, we then obtain
    \begin{align*}
      W^{G}_{p,A\cup B}({X}_{A\cup B}) & = \widehat{V}^{G}_{A\cup B}({X}_{A\cup B})
      - \sum_{a \subseteq A, a\neq\emptyset}W^{G}_{p,a}({X}_{a})
      - \sum_{b \subseteq B, b\neq\emptyset}W^{G}_{p,b}({X}_{b})\\
      & \quad - \sum_{\substack{a \subseteq A, b \subseteq B\\ a\neq\emptyset, b\neq
          \emptyset, |a\cup b|<|A\cup B|}} W^{G}_{p,a\cup b}({X}_{a\cup b})
      - W^{G}_{p,\emptyset}\\
      &= \widehat{V}^{G}_{p,A}({X}_{A}) + \widehat{V}^{G}_{p,B}({X}_{B})\\
      & \quad - \sum_{a \subseteq A}W^{G}_{p,a}({X}_{a})
      - \sum_{b \subseteq B}W^{G}_{p,b}({X}_{b})
      - \sum_{\substack{a \subseteq A, b \subseteq B\\ a\neq\emptyset, b\neq
          \emptyset, |a\cup b|<|A\cup B|}} W^{G}_{p, a\cup b}({X}_{a\cup b}).
    \end{align*}
    Now, we apply the induction
    hypothesis to each $W^{G}_{p,A\cup B}$: $|a\cup b| < |A\cup B|
    \leq |{P}| = n+1$ and finally obtain
    \begin{align*}
      W^{G}_{p,A\cup B}({X}_{A\cup B}) & = \widehat{V}^{G}_{p,A}({X}_{A}) - \sum_{a \subseteq A}W^{G}_{p,a}({X}_{a}) + \widehat{V}^{G}_{p,B}({X}_{B}) - \sum_{b \subseteq B}W^{G}_{p,b}({X}_{b})\\
      & \quad - \sum_{\substack{a \subseteq A, b \subseteq B \\ a\neq\emptyset, b\neq \emptyset, |a\cup b|<|A\cup B|}} W^{G}_{p,a\cup b}({X}_{a\cup b})\\
      &= - \sum_{\substack{a \subseteq A, b \subseteq B \\ a\neq\emptyset, b\neq
          \emptyset, |a\cup b|<|A\cup B|}} W^{G}_{p, a\cup b}({X}_{a\cup b}) = 0.
    \end{align*}
  \end{proof}
\section{Example (small linear molecule)}\label{sec:LinMolExample}
In the following, let us shortly give an example corresponding to a
linear molecule of length three. Let $G=(\{1,2,3\},\{\{1,2\},
\{2,3\}\}$. Then, from (\ref{equ:w_pu_contribution}) and
(\ref{equ:Wpu}) it results for level $p=0$:
\begin{align*}
  \myomega_{0,\{1\}}(X_{\{1\}})&:=\widehat{V}_0(X_{\{1\}}),\\
  \myomega_{0,\{2\}}(X_{\{2\}})&:=\widehat{V}_0(X_{\{2\}}),\\
  \myomega_{0,\{3\}}(X_{\{3\}})&:=\widehat{V}_0(X_{\{3\}}),\\
\myomega_{0,\{1,2\}}(X_{\{1,2\}})&:=\widehat{V}_0(X_{\{1,2\}})-(\myomega_{0,\{1\}}(X_{\{1\}})+\myomega_{0,\{2\}}(X_{\{2\}}))\\
\myomega_{0,\{2,3\}}(X_{\{2,3\}})&:=\widehat{V}_0(X_{\{2,3\}})-(\myomega_{0,\{2\}}(X_{\{2\}})+\myomega_{0,\{3\}}(X_{\{3\}})),\\
\myomega_{0,\{1,2,3\}}(X_{\{1,2,3\}})&:=\widehat{V}_0(X_{\{1,2,3\}})-(\myomega_{0,\{2,3\}}(X_{\{2,3\}})+\myomega_{0,\{1,2\}}(X_{\{1,2\}})+\myomega_{0,\{1\}}+\myomega_{0,\{2\}}+\myomega_{0,\{3\}}).
\end{align*}
For level $p=1$ we obtain:
\begin{align*}
  \myomega_{1,\{1\}}(X_{\{1\}}):=\widehat{V}_1(X_{\{1\}})&-\myomega_{0,\{1\}}(X_{\{1\}}),\\
  \myomega_{1,\{2\}}(X_{\{2\}}):=\widehat{V}_1(X_{\{2\}})&-\myomega_{0,\{2\}}(X_{\{2\}}),\\
  \myomega_{1,\{3\}}(X_{\{3\}}):=\widehat{V}_1(X_{\{3\}})&-\myomega_{0,\{3\}}(X_{\{3\}}),\\
  \myomega_{1,\{1,2\}}(X_{\{1,2\}}):=\widehat{V}_1(X_{\{1,2\}})&-(\myomega_{1,\{1\}}(X_{\{1\}})+\myomega_{1,\{2\}}(X_{\{2\}}))-(\myomega_{1,\{1\}}(X_{\{1\}})+\myomega_{1,\{2\}}(X_{\{2\}}))\\
  \myomega_{1,\{2,3\}}(X_{\{2,3\}}):=\widehat{V}_1(X_{\{2,3\}})&-(\myomega_{1,\{2\}}(X_{\{2\}})+\myomega_{1,\{3\}}(X_{\{3\}}))-(\myomega_{0,\{2\}}(X_{\{2\}})+\myomega_{0,\{3\}}(X_{\{3\}})),\\
  \myomega_{0,\{1,2,3\}}(X_{\{1,2,3\}}):=\widehat{V}_1(X_{\{1,2,3\}})&-(\myomega_{1,\{2,3\}}(X_{\{2,3\}})+\myomega_{1,\{1,2\}}(X_{\{1,2\}})+\myomega_{1,\{1\}}+\myomega_{1,\{2\}})\\
  &-(\myomega_{0,\{2,3\}}(X_{\{2,3\}})+\myomega_{0,\{1,2\}}(X_{\{1,2\}})+\myomega_{0,\{1\}}+\myomega_{0,\{2\}}).
\end{align*}
Note here that the induced subgraph $G_{\{1,3\}}$ is not
connected. Hence, the contributions $\myomega_{p,\{1,3\}}$ are zero for all
$p\in\mathbb{N}$.

\end{appendix}

\bibliographystyle{plain}
\bibliography{Literature}

\begin{thebibliography}{10}

\bibitem{Abell1985}
G.~Abell.
\newblock Empirical chemical pseudopotential theory of molecular and metallic
  bonding.
\newblock {\em Physical Review B}, 31(10):6184--6196, 1985.

\bibitem{ambrosio2011semiclassical}
L.~Ambrosio, A.~Figalli, G.~Friesecke, J.~Giannoulis, and T.~Paul.
\newblock Semiclassical limit of quantum dynamics with rough potentials and
  well-posedness of transport equations with measure initial data.
\newblock {\em Communications on Pure and Applied Mathematics},
  64(9):1199--1242, 2011.

\bibitem{Amovilli2002}
A.~Amovilli, I.~Cacelli, S.~Campanile, and G.~Prampolini.
\newblock Calculation of the intermolecular energy of large molecules by a
  fragmentation scheme: Application to the {4-n-pentyl-4-cyanobiphenyl (5CB)}
  dimer.
\newblock {\em Journal of Chemical Physics}, 117:3003--3012, 2002.

\bibitem{andre2014nobel}
J.-M. Andr{\'e}.
\newblock The {N}obel {P}rize in chemistry 2013.
\newblock {\em Chemistry International}, 36(2):2--7, 2014.

\bibitem{bates2011development}
D.~Bates, J.~Smith, T.~Janowski, and G.~Tschumper.
\newblock Development of a 3-body: many-body integrated fragmentation method
  for weakly bound clusters and application to water clusters {($H_2O$)} n =
  3-10, 16, 17.
\newblock {\em The Journal of Chemical Physics}, 135(4):044123, 2011.

\bibitem{beran2009approximating}
G.~Beran.
\newblock Approximating quantum many-body intermolecular interactions in
  molecular clusters using classical polarizable force fields.
\newblock {\em The Journal of Chemical Physics}, 130(16):164115, 2009.

\bibitem{bowler2012mathcal}
D.~Bowler and T.~Miyazaki.
\newblock {O(N)} methods in electronic structure calculations.
\newblock {\em Reports on Progress in Physics}, 75(3):036503, 2012.

\bibitem{Brenner2002}
D.~Brenner.
\newblock A second-generation reactive bond order ({REBO}) potential energy
  expression for hydrocarbons.
\newblock {\em Journal of Physics: Condensed Matter}, 14:783--802, 2002.

\bibitem{le2005computational}
C.~Le Bris.
\newblock Computational chemistry from the perspective of numerical analysis.
\newblock {\em Acta Numerica}, 14:363--444, 2005.

\bibitem{bungartz2004sparse}
H-J. Bungartz and M.~Griebel.
\newblock Sparse grids.
\newblock {\em Acta Numerica}, 13:147--269, 2004.

\bibitem{Bungartz.Griebel.Roschke.ea:1996}
H.-J. Bungartz, M.~Griebel, D.~R\"oschke, and C.~Zenger.
\newblock A proof of convergence for the combination technique for the
  {L}aplace equation using tools of symbolic computation.
\newblock {\em Mathematics and Computers in Simulation}, 42:595--605, 1996.

\bibitem{bygrave2012embedded}
P.~Bygrave, N.~Allan, and F.~Manby.
\newblock The embedded many-body expansion for energetics of molecular
  crystals.
\newblock {\em The Journal of Chemical Physics}, 137(16):164102, 2012.

\bibitem{cances2008molecular}
E.~Canc{\`e}s, C.~Le Bris, and P-L. Lions.
\newblock Molecular simulation and related topics: some open mathematical
  problems.
\newblock {\em Nonlinearity}, 21(9):T165--T176, 2008.

\bibitem{cances2003computational}
E.~Canc{\`e}s, M.~Defranceschi, W.~Kutzelnigg, C.~Le Bris, and Y.~Maday.
\newblock Computational {Q}uantum {C}hemistry: A primer.
\newblock {\em Handbook of Numerical Analysis}, 10:3--270, 2003.

\bibitem{carter1997vibrational}
S.~Carter, S.~Culik, and J.~Bowman.
\newblock Vibrational self-consistent field method for many-mode systems: A new
  approach and application to the vibrations of {CO} adsorbed on {Cu} (100).
\newblock {\em The Journal of Chemical Physics}, 107(24):10458--10469, 1997.

\bibitem{Collins2006}
M.~Collins and V.~Deev.
\newblock Accuracy and efficiency of electronic energies from systematic
  molecular fragmentation.
\newblock {\em Journal of Chemical Physics}, 125:104104, 2006.

\bibitem{dahlke2007electrostatically}
E.~Dahlke and D.~Truhlar.
\newblock Electrostatically embedded many-body correlation energy, with
  applications to the calculation of accurate second-order {M{\o}ller-Plesset}
  perturbation theory energies for large water clusters.
\newblock {\em Journal of Chemical Theory and Computation}, 3(4):1342--1348,
  2007.

\bibitem{DawBaskes1984}
M.~Daw and M.~Baskes.
\newblock Embedded-atom method: Derivation and application to impurities,
  surfaces and other defects in metals.
\newblock {\em Physical Review B}, 29(12):6443--6453, 1984.

\bibitem{Deev2005}
V.~Deev and M.~A. Collins.
\newblock Approximate ab initio energies by systematic molecular fragmentation.
\newblock {\em Journal of Chemical Physics}, 122(15):154102, 2005.

\bibitem{deyonker2006correlation}
N.~DeYonker, T.~Cundari, and A.~Wilson.
\newblock The correlation consistent composite approach {(ccCA)}: An
  alternative to the {Gaussian-n} methods.
\newblock {\em The Journal of Chemical Physics}, 124(11):114104, 2006.

\bibitem{dunning1989gaussian}
T.~Dunning.
\newblock Gaussian basis sets for use in correlated molecular calculations.
  {I}. {T}he atoms boron through neon and hydrogen.
\newblock {\em The Journal of Chemical Physics}, 90(2):1007--1023, 1989.

\bibitem{elliott2010partition}
P.~Elliott, K.~Burke, M.~Cohen, and A.~Wasserman.
\newblock Partition density-functional theory.
\newblock {\em Physical Review A}, 82(2):024501, 2010.

\bibitem{fast1999infinite}
P.~Fast, M.~S{\'a}nchez, and D.~Truhlar.
\newblock Infinite basis limits in electronic structure theory.
\newblock {\em The Journal of Chemical Physics}, 111(7):2921--2926, 1999.

\bibitem{feller1993use}
D.~Feller.
\newblock The use of systematic sequences of wave functions for estimating the
  complete basis set, full configuration interaction limit in water.
\newblock {\em The Journal of Chemical Physics}, 98(9):7059--7071, 1993.

\bibitem{field1990combined}
M.~Field, P.~Bash, and M.~Karplus.
\newblock A combined quantum mechanical and molecular mechanical potential for
  molecular dynamics simulations.
\newblock {\em Journal of Computational Chemistry}, 11(6):700--733, 1990.

\bibitem{friesecke2003multiconfiguration}
G.~Friesecke.
\newblock The multiconfiguration equations for atoms and molecules: charge
  quantization and existence of solutions.
\newblock {\em Archive for Rational Mechanics and Analysis}, 169(1):35--71,
  2003.

\bibitem{Gerstner.Griebel:1998}
T.~Gerstner and M.~Griebel.
\newblock Numerical integration using sparse grids.
\newblock {\em Numer. Algorithms}, 18:209--232, 1998.

\bibitem{gerstner2003dimension}
T.~Gerstner and M.~Griebel.
\newblock Dimension--adaptive tensor--product quadrature.
\newblock {\em Computing}, 71(1):65--87, 2003.

\bibitem{Goedecker1999}
S.~Goedecker.
\newblock Linear scaling electronic structure methods.
\newblock {\em Reviews of Modern Physics}, 71(4):1085--1123, 1999.

\bibitem{gora2011interaction}
U.~G{\'o}ra, R.~Podeszwa, W.~Cencek, and K.~Szalewicz.
\newblock Interaction energies of large clusters from many-body expansion.
\newblock {\em The Journal of Chemical Physics}, 135(22):224102, 2011.

\bibitem{gordon2012fragmentation}
M.~Gordon, D.~Fedorov, S.~Pruitt, and L.~Slipchenko.
\newblock Fragmentation methods: a route to accurate calculations on large
  systems.
\newblock {\em Chemical Reviews}, 112(1):632--672, 2012.

\bibitem{GreshClaverie1984}
N.~Gresh, P.~Claverie, and A.~Pullman.
\newblock Theoretical studies of molecular conformation. {D}erivation of an
  additive procedure for the computation of intramolecular interaction
  energies. {C}omparison with ab-initio {SCF} computations.
\newblock {\em Theoretica Chimica Acta}, 66:1--20, 1984.

\bibitem{griebel2007sparse}
M.~Griebel and J.~Hamaekers.
\newblock Sparse grids for the {S}chr{\"o}dinger equation.
\newblock {\em ESAIM: Mathematical Modelling and Numerical Analysis},
  41(02):215--247, 2007.

\bibitem{griebel2010tensor}
M.~Griebel and J.~Hamaekers.
\newblock Tensor product multiscale many-particle spaces with finite-order
  weights for the electronic {S}chr{\"o}dinger equation.
\newblock {\em Zeitschrift f{\"u}r Physikalische Chemie}, 224(3-4):527--543,
  2010.

\bibitem{griebel2014bond}
M.~Griebel, J.~Hamaekers, and F.~Heber.
\newblock A bond order dissection {ANOVA} approach for efficient electronic
  structure calculations.
\newblock In {\em Extraction of Quantifiable Information from Complex Systems},
  pages 211--235. Springer, 2014.

\bibitem{Griebel.Harbrecht:2014}
M.~Griebel and H.~Harbrecht.
\newblock On the convergence of the combination technique.
\newblock In {\em Sparse grids and Applications}, volume~97 of {\em Lecture
  Notes in Computational Science and Engineering}, pages 55--74. Springer,
  2014.

\bibitem{Griebel2007a}
M.~Griebel, S.~Knapek, and G.~Zumbusch.
\newblock {\em {Numerical Simulation in Molecular Dynamics -- Numerics,
  Algorithms, Parallelization, Applications}}.
\newblock Springer-Verlag, Heidelberg, 2007.

\bibitem{halkier1998basis}
A.~Halkier, T.~Helgaker, P.~J{\o}rgensen, W.~Klopper, H.~Koch, J.~Olsen, and
  A.~Wilson.
\newblock Basis-set convergence in correlated calculations on {N}e,
  {{N}}$_{2}$, and {H}$_2${O}.
\newblock {\em Chemical Physics Letters}, 286(3):243--252, 1998.

\bibitem{Hamaekers:2010}
J.~Hamaekers.
\newblock {\em {Sparse Grids for the Electronic Schr{\"o}dinger Equation:
  Construction and Application of Sparse Tensor Product Multiscale
  Many-Particle Spaces with Finite-Order Weights for Schr{\"o}dinger's
  Equation}}.
\newblock S{\"u}dwestdeutscher Verlag f{\"u}r Hochschulschriften,
  Saarbr{\"u}cken, 2010.

\bibitem{Hayes2006}
M.~Y. Hayes, B.~Li, and H.~Rabitz.
\newblock Estimation of molecular properties by high-dimensional model
  representation.
\newblock {\em Journal of Physical Chemistry}, 110:264--272, 2006.

\bibitem{molecuilder}
F.~Heber.
\newblock {MoleCuilder - a molecular builder}.
\newblock \url{https://trac.ins.uni-bonn.de/projects/molecuilder}.

\bibitem{Heber2014}
F.~Heber.
\newblock {\em {Ein systematischer, linear skalierender Fragmentansatz f\"{u}r
  das Elektronenstukturproblem}}.
\newblock PhD thesis, Intitut f\"ur Numerische Simulation, Rheinische
  Friedrich-Wilhelms-Universit\"{a}t Bonn, 2014.

\bibitem{helgaker1997basis}
T.~Helgaker, W.~Klopper, H.~Koch, and J.~Noga.
\newblock Basis-set convergence of correlated calculations on water.
\newblock {\em The Journal of Chemical Physics}, 106(23):9639--9646, 1997.

\bibitem{huang2011quantum}
C.~Huang, M.~Pavone, and E.~Carter.
\newblock Quantum mechanical embedding theory based on a unique embedding
  potential.
\newblock {\em The Journal of Chemical Physics}, 134(15):154110, 2011.

\bibitem{janssen2004massively}
C.~Janssen, I.~Nielsen, M.~Leininger, E.~Valeev, and E.~Seidl.
\newblock The {M}assively {P}arallel {Q}uantum {C}hemistry {P}rogram ({MPQC}).
\newblock {\em Sandia National Laboratories, Livermore, CA, USA}, 2004.

\bibitem{jensen2005estimating}
F.~Jensen.
\newblock Estimating the {Hartree-Fock} limit from finite basis set
  calculations.
\newblock {\em Theoretical Chemistry Accounts}, 113(5):267--273, 2005.

\bibitem{jensen2013introduction}
F.~Jensen.
\newblock {\em Introduction to Computational Chemistry}.
\newblock John Wiley \& Sons, 2013.

\bibitem{jurevcka2006benchmark}
P.~Jure{\v{c}}ka, J.~{\v{S}}poner, J.~{\v{C}}ern{\`y}, and P.~Hobza.
\newblock Benchmark database of accurate ({MP2} and {CCSD (T)} complete basis
  set limit) interaction energies of small model complexes, {DNA} base pairs,
  and amino acid pairs.
\newblock {\em Physical Chemistry Chemical Physics}, 8(17):1985--1993, 2006.

\bibitem{Kitaura1999}
K.~Kitaura, E.~Ikeo, T.~Asada, T.~Nakano, and M.~Uebayasi.
\newblock {Fragment molecular orbital method: an approximate computational
  method for large molecules}.
\newblock {\em Chemical Physics Letters}, 313:701--706, 1999.

\bibitem{kohn1996density}
W.~Kohn.
\newblock Density functional and density matrix method scaling linearly with
  the number of atoms.
\newblock {\em Physical Review Letters}, 76(17):3168, 1996.

\bibitem{komeiji2003fragment}
Y.~Komeiji, T.~Nakano, K.~Fukuzawa, Y.~Ueno, Y.~Inadomi, T.~Nemoto,
  M.~Uebayasi, D.~Fedorov, and K.~Kitaura.
\newblock Fragment molecular orbital method: application to molecular dynamics
  simulation, ab initio {FMO-MD}.
\newblock {\em Chemical Physics Letters}, 372(3):342--347, 2003.

\bibitem{kreusler2012mixed}
H.-C. Kreusler and H.~Yserentant.
\newblock The mixed regularity of electronic wave functions in fractional order
  and weighted {S}obolev spaces.
\newblock {\em Numerische Mathematik}, 121(4):781--802, 2012.

\bibitem{lao2016understanding}
K.~Lao, K.-Y. Liu, R.~Richard, and J.~Herbert.
\newblock Understanding the many-body expansion for large systems. {II.
  Accuracy} considerations.
\newblock {\em The Journal of Chemical Physics}, 144(16):164105, 2016.

\bibitem{lubich2008quantum}
C.~Lubich.
\newblock {\em {From Quantum to Classical Molecular Dynamics: Reduced Models
  and Numerical Analysis}}.
\newblock European Mathematical Society, 2008.

\bibitem{martin1999ab}
J.~Martin.
\newblock Ab initio thermochemistry beyond chemical accuracy for first-and
  second-row compounds.
\newblock In {\em Energetics of Stable Molecules and Reactive Intermediates},
  pages 373--415. Springer, 1999.

\bibitem{MarxHutter2000}
D.~Marx and J.~Hutter.
\newblock Ab initio molecular dynamics: Theory and implementation.
\newblock In {\em {Modern Methods and Algorithms of Quantum Chemistry}},
  volume~1 of {\em NIC Series}, pages 301--440. Forschungszentrum Juelich,
  Deutschland, 2000.

\bibitem{MaserasMorokuma1995}
F.~Maseras and K.~Morokuma.
\newblock {IMOMM} - a new integrated ab-initio plus molecular mechanics
  geometry optimization scheme of equilibrium structures and transition-states.
\newblock {\em Journal of Computational Chemistry}, 16(9):1170--1179, 1995.

\bibitem{mayhall2011molecules}
N.~Mayhall and K.~Raghavachari.
\newblock Molecules-in-molecules: An extrapolated fragment-based approach for
  accurate calculations on large molecules and materials.
\newblock {\em Journal of Chemical Theory and Computation}, 7(5):1336--1343,
  2011.

\bibitem{nakano2001multiscale}
A.~Nakano, M.~Bachlechner, R.~Kalia, E.~Lidorikis, P.~Vashishta, G.~Voyiadjis,
  T.~Campbell, S.~Ogata, and F.~Shimojo.
\newblock Multiscale simulation of nanosystems.
\newblock {\em Computing in Science \& Engineering}, 3(4):56--66, 2001.

\bibitem{nobile2015convergence}
F.~Nobile, L.~Tamellini, and R.~Tempone.
\newblock Convergence of quasi-optimal sparse-grid approximation of
  {H}ilbert-space-valued functions: application to random elliptic {PDEs}.
\newblock {\em Numerische Mathematik}, pages 1--46, 2015.

\bibitem{ohlinger2009efficient}
W.~Ohlinger, P.~Klunzinger, B.~Deppmeier, and W.~Hehre.
\newblock Efficient calculation of heats of formation.
\newblock {\em The Journal of Physical Chemistry A}, 113(10):2165--2175, 2009.

\bibitem{peterson2002accurate}
K.~Peterson and T.~Dunning.
\newblock Accurate correlation consistent basis sets for molecular
  core--valence correlation effects: The second row atoms {Al--Ar}, and the
  first row atoms {B--Ne} revisited.
\newblock {\em The Journal of Chemical Physics}, 117(23):10548--10560, 2002.

\bibitem{peterson1994benchmark}
K.~Peterson, D.~Woon, and T.~Dunning.
\newblock Benchmark calculations with correlated molecular wave functions.
  {IV.} the classical barrier height of the {$H + H_2 \to H_2 + H$} reaction.
\newblock {\em The Journal of Chemical Physics}, 100(10):7410--7415, 1994.

\bibitem{praprotnik2008multiscale}
M.~Praprotnik, L.~Site, and K.~Kremer.
\newblock Multiscale simulation of soft matter: From scale bridging to adaptive
  resolution.
\newblock {\em Annu. Rev. Phys. Chem.}, 59:545--571, 2008.

\bibitem{prodan2005nearsightedness}
E.~Prodan and W.~Kohn.
\newblock Nearsightedness of electronic matter.
\newblock {\em Proceedings of the National Academy of Sciences of the United
  States of America}, 102(33):11635--11638, 2005.

\bibitem{richard2012generalized}
R.~Richard and J.~Herbert.
\newblock A generalized many-body expansion and a unified view of
  fragment-based methods in electronic structure theory.
\newblock {\em The Journal of Chemical Physics}, 137(6):064113, 2012.

\bibitem{richard2014understanding}
R.~Richard, K.~Lao, and J.~Herbert.
\newblock Understanding the many-body expansion for large systems. {I.
  Precision} considerations.
\newblock {\em The Journal of Chemical Physics}, 141(1):014108, 2014.

\bibitem{schneider2009analysis}
R.~Schneider.
\newblock Analysis of the projected coupled cluster method in electronic
  structure calculation.
\newblock {\em Numerische Mathematik}, 113(3):433--471, 2009.

\bibitem{senn2009qm}
H.~Senn and W.~Thiel.
\newblock {QM/MM} methods for biomolecular systems.
\newblock {\em Angewandte Chemie International Edition}, 48(7):1198--1229,
  2009.

\bibitem{shimojo2008divide}
F.~Shimojo, R.~Kalia, A.~Nakano, and P.~Vashishta.
\newblock Divide-and-conquer density functional theory on hierarchical
  real-space grids: Parallel implementation and applications.
\newblock {\em Physical Review B}, 77(8):085103, 2008.

\bibitem{Skylaris2005}
C.-K. Skylaris, P.~Haynes, A.~Mostofi, and M.~Payne.
\newblock {Introducing ONETEP: Linear-scaling density functional simulations on
  parallel computers.}
\newblock {\em Journal of Chemical Physics}, 122(8):84119, 2005.

\bibitem{Tersoff1989}
J.~Tersoff.
\newblock Modeling solid-state chemistry: Interatomic potentials for
  multicomponent systems.
\newblock {\em Physical Review B}, 39:5566--5568, 1989.

\bibitem{truhlar1998basis}
D.~Truhlar.
\newblock Basis-set extrapolation.
\newblock {\em Chemical Physics Letters}, 294(1):45--48, 1998.

\bibitem{tschumper2006multicentered}
G.~Tschumper.
\newblock Multicentered integrated {QM:QM} methods for weakly bound clusters:
  An efficient and accurate 2-body: many-body treatment of hydrogen bonding and
  van der {W}aals interactions.
\newblock {\em Chemical Physics Letters}, 427(1):185--191, 2006.

\bibitem{Vreven2000}
T.~Vreven and K.~Morokuma.
\newblock On the application of the {IMOMO} (integrated molecular orbital +
  molecular orbital) method.
\newblock {\em Journal of Computational Chemistry}, 21(16):1419--1432, 2000.

\bibitem{wang2014divide}
L.-W. Wang.
\newblock Divide-and-conquer quantum mechanical material simulations with
  exa\-scale supercomputers.
\newblock {\em National Science Review}, 1(4):604--617, 2014.

\bibitem{wen2012practical}
S.~Weno, K.~Nanda, Y.~Huang, and G.~Beran.
\newblock Practical quantum mechanics-based fragment methods for predicting
  molecular crystal properties.
\newblock {\em Physical Chemistry Chemical Physics}, 14(21):7578--7590, 2012.

\bibitem{wilson1996gaussian}
A.~Wilson, T.~van Mourik, and T.~Dunning.
\newblock Gaussian basis sets for use in correlated molecular calculations.
  {VI}. {S}extuple zeta correlation consistent basis sets for boron through
  neon.
\newblock {\em Journal of Molecular Structure: THEOCHEM}, 388:339--349, 1996.

\bibitem{yang1995density}
W.~Yang and T.-S. Lee.
\newblock A density-matrix divide-and-conquer approach for electronic structure
  calculations of large molecules.
\newblock {\em The Journal of Chemical Physics}, 103(13):5674--5678, 1995.

\bibitem{yserentant2014regularity}
H.~Yserentant.
\newblock Regularity, complexity, and approximability of electronic
  wavefunctions.
\newblock In {\em Extraction of Quantifiable Information from Complex Systems},
  pages 413--428. Springer, 2014.

\bibitem{yserentant2010regularity}
Harry Yserentant.
\newblock {\em Regularity and Approximability of Electronic Wave Functions}.
\newblock Springer, 2010.

\bibitem{zeiser2012wavelet}
A.~Zeiser.
\newblock Wavelet approximation in weighted {S}obolev spaces of mixed order
  with applications to the electronic {S}chr{\"o}dinger equation.
\newblock {\em Constructive Approximation}, 35(3):293--322, 2012.

\bibitem{rezac2009multilevel}
J.~Řezáč and D.~Salahub.
\newblock Multilevel fragment-based approach {(MFBA)}: a novel hybrid
  computational method for the study of large molecules.
\newblock {\em Journal of Chemical Theory and Computation}, 6(1):91--99, 2009.

\end{thebibliography}

\end{document}